\documentclass
{amsart}

\usepackage{amssymb,amscd}
\usepackage[all,line,arc,curve,color,frame,pdf]{xy}
\usepackage{tikz}
\usepackage{tikz-cd}
\usetikzlibrary{positioning, trees, snakes}
\usepackage[textsize=tiny]{todonotes}
\usepackage{hyperref}
\usepackage[shortlabels]{enumitem}
\usepackage{float}
\usepackage[paper=a4paper, margin=3.5cm]{geometry}
\usepackage[symbol]{footmisc}
\numberwithin{equation}{subsection}
\theoremstyle{plain}
\newtheorem{thm}{Theorem}[section]
\newtheorem*{thstar}{Theorem}
\newtheorem{cor}[thm]{Corollary}
\newtheorem{lemma}[thm]{Lemma}
\newtheorem{prop}[thm]{Proposition}

\newtheorem{question}[thm]{Question}

\theoremstyle{remark}
\newtheorem{rem}[thm]{Remark}
\newtheorem{ex}[thm]{Example}
\newtheorem{conj}[thm]{Conjecture}

\theoremstyle{definition}
\newtheorem{defi}[thm]{Definition}

\def\GM{\overline{Q}} 
\def\LG{LG}
\def\Pe{P}
\def\Grass{Grass}
\def\Mbar{\overline{{\mathcal M}}}
\def\LLL{\Lambda}
\def\PD{PD}

\def\TTT{\mathbb{T}}
\def\Mo{{\mathcal M}_\alpha}
\newcommand\Ls{\Lambda}
\newcommand\Lb{{\mathcal{L}}}

\DeclareMathOperator\grad{grad}

\subjclass[2020]{primary: 62R01, 14M17, 14N10, 14E05 secondary: 14C17, 14M15, 14Q15, 60G15}
\usepackage[textsize=tiny]{todonotes}

\usepackage{enumitem}

\newcommand\ignore[1]{}

\DeclareMathOperator{\HH}{H}
\DeclareMathOperator{\pt}{pt}

\newcommand\CC{{\mathbb{C}}}
\newcommand\PP{{\mathbb{P}}}
\newcommand\RR{{\mathbb{R}}}

\newcommand\ZZ{{\mathbb{Z}}}

\newcommand\SSS{{\mathbb{S}}}

\def\C{{\mathbb C}}

\def\P{{\mathbb P}}

\def\cJ{{\mathcal J}}

\def\cO{{\mathcal{O}}}

\def\operatorname#1{\mathop{\rm #1}\nolimits}

\def\Hom{\operatorname{Hom}}

\def\Hom{\operatorname{Hom}}

\def\deg{\operatorname{deg}}

\newcommand{\Chi}{\ensuremath \raisebox{2pt}{$\chi$}}
\newcommand{\pb}{\ar@{}[dr]|{\text{\pigpenfont J}}}

\makeatletter
\newcommand{\xleftrightarrow}[2][]{\ext@arrow 3359\leftrightarrowfill@{#1}{#2}}
\makeatother
\newcommand{\xdasharrow}[2][->]{
\tikz[baseline=-\the\dimexpr\fontdimen22\textfont2\relax]{
\node[anchor=south,font=\scriptsize, inner ysep=1.5pt,outer xsep=2.2pt](x){#2};
\draw[shorten <=3.4pt,shorten >=3.4pt,dashed,#1](x.south west)--(x.south east);
}}

\providecommand{\symdif}{\mathbin{\mathpalette\xdotminus\relax}}
\newcommand{\xdotminus}[2]{%
  \ooalign{\hidewidth$\vcenter{\hbox{$#1\dot{}$}}$\hidewidth\cr$#1-$\cr}%
}

\newcommand\iso{{\ \cong\ }}

\begin{document}
\title{Maximum likelihood degree, complete quadrics and ${\mathbb C}^*$-action}

\author[Micha{\l}ek]{Mateusz Micha{\l}ek}
\address{Max Planck Institute for Mathematics in the Sciences,
Leipzig, Germany
}
\author[Monin]{Leonid Monin}
\address{University of Bristol, School of Mathematics, BS8 1TW, Bristol, UK}
\author[Wi{\'s}niewski]{Jaros{\l}aw A. Wi{\'s}niewski}
\address{Istitute of Mathematics,
University of Warsaw, 
PL-02-097 Warszawa
}

\begin{abstract}
We study the maximum likelihood (ML) degree of linear concentration models in algebraic statistics. We relate it to an intersection problem on the variety of complete quadrics. This allows us to provide an explicit, basic, albeit of high computational complexity, formula for the ML-degree.  The variety of complete quadrics is an exact analog for symmetric matrices of the permutohedron variety for the diagonal matrices.
\end{abstract}

\thanks{\noindent The project has been supported by Polish National Science
  Center grant 2016/23/G/ST1/ 04828 and EPSRC Early Career Fellowship EP/R023379/1. 
	JW thanks MPI MiS for hospitality. }
\maketitle

\section{Introduction}
With this article we would like to initiate  the program of studying inversion of matrices by the geometry of $\CC^*$-actions on homogeneous varieties. Our main motivations come from algebraic statistics, precisely linear concentration models, based on the multivariate Gaussian model. We present the main question in two equivalent versions:

\begin{itemize}[leftmargin=*]
\item {\it Given a general subspace $\Ls$ of the space of symmetric matrices, what is the degree of the variety obtained by inverting all matrices in $\Ls$?}
\item {\it What is the maximum likelihood (ML) degree of a general linear concentration model?}
\end{itemize}

The computation of the maximum likelihood degree for linear concentration model has recently attracted a lot of attention, consult e.g.~\cite{StUh, SethBook, uhler2012geometry, sturmfels2019estimating}. Our main new input is to relate this invariant to the \emph{variety of complete quadrics} $\GM$. We show that it is a smooth resolution of the graph of the rational map of inversion of symmetric matrices represented as elements in $\PP(S^2V)$, the projectivized symmetric square of a vector space $V$. 
\begin{thstar}[Corollary \ref{cor:gaussmoduli}]
The variety $\GM$ is a smooth projective variety, which contains an open part parametrizing general orbits of a special $\CC^*$-action on the Grassmannian of Lagrangian spaces in $V\oplus V^*$. It comes with two natural morphisms $\pi_0:\GM\to  \PP(S^2V)$ and $\pi_\infty:\GM\to  \PP(S^2V^*)$ which resolve the inversion map $\Phi: \PP(S^2V)\dashrightarrow \PP(S^2V^*)$. Moreover, $\GM$ admits an algebraic action of $SL(V)$ which makes the above morphisms equivariant.
\end{thstar}

The study of the variety $\GM$ is a very classical topic in algebraic geometry going back to the fundamental results of Schubert \cite{Schubert1879} from 19-th century.  The cohomological properties of the variety of complete quadrics were addressed more recently in \cite{bifet1990cohomology, DeConciniProcesi1, DeConciniProcesi2, LaksovCompleteQuadrics}.

Therefore we show far reaching relations between the new developments in algebraic statistics and the classical geometry of complete quadrics. This leads to new, effective methods. In particular, the computation of the ML-degree translates directly to an enumerative intersection problem on $\GM$. We address it using the action of a Cartan torus in the group $SL(V)$; in particular we use equivariant $K$-theory and cohomology. Our results provide a basic formula for the ML-degree.

\begin{thstar}[cf.~Corollary \ref{cor:sum}]
The $ML$-degree is a sum over the finite number of torus fixed points on the variety $\GM$ of rational numbers assigned to each of them. Both, the enumeration of such fixed points (Proposition \ref{prop:fixedpoints}), and computation of the rational numbers (Proposition \ref{prop:orbits} and Corollary \ref{cor:sum}) is completely explicit.
\end{thstar}
Examples showing how the torus fixed points look like and how to obtain the rational numbers are presented in \ref{GMn=2}, \ref{GMn=3}, \ref{ex:n=3} and \ref{GMn=3.2}.

The study of the cohomology class of the graph of a rational map is also an active area of research in different branches of mathematics \cite{MSUZ, cid2020mixed}, \cite[7.1.3]{dolgachev2012classical} with results going back to Cremona.

We would like to point out that our approach is quite similar to modern developments in the theory of matroids \cite{JunePhD, JuneKarim, JuneRev}. Explicitly, a projective space $\PP^{n-1}$ may be identified with the projectivisation of the space of $n\times n$ diagonal matrices. The inversion map is the well-known classical Cremona transformation and a linear subspace of $\PP^{n-1}$ may be identified with a representable matroid. In algebraic statistics this corresponds to toric models. In that case, a particular role is played by the permutohedral variety, which is a smooth resolution of the graph of the Cremona map. In our setting we replace $\PP^{n-1}$ by the projective space of symmetric matrices. Then the analog of the permutohedral variety is played by $\GM$. In that analogy, uniform matroids and their beta invariants correspond to general spaces of symmetric matrices and the ML-degree. We hope that further study of the cohomology of $\GM$ will reveal more similarities.

The plan of the article is as follows. We study $\CC^*$-actions on smooth, convex, projective varieties \cite[0.4]{FP}, with a focus on the Lagrangian Grassmannian $\LG$. We realize the moduli spaces we study as subvarieties of Kontsevich spaces. Using \cite[Theorem 2]{FP} we obtain smooth, projective varieties. In types A,C and D this parallels the constructions of Thaddeus \cite{thaddeus1999complete} of varieties of \emph{complete collineations}, \emph{complete quadrics} and \emph{complete skew forms}.
More precisely, $\GM$ is a compactification of a space of orbits of an action of $\CC^*$ on $\LG$ which gives rise to the Cremona transformation, \ref{action=>inversion_of_matrices}. The description of the $K$-theory we provide through so-called moment or GKM graphs is very explicit and down-to-earth. Finally, we obtain the formula for the ML-degree just in terms of basic arithmetic operations on rational numbers. This allows us to predict many formulas for the ML-degree, cf.~Conjecture \ref{conj:form}. We implement our algorithms in Sage and Macaulay2.

We hope that our article will be of interest both to algebraic geometers and algebraic statisticians. The variety of complete quadrics has played an important role in many enumerative problems in pure algebraic geometry, cf.~\cite{laksov1987completed} and reference therein. However, to our knowledge it seems not to have been applied in algebraic statistics so far.
In this article we focus on introducing these methods and the computational aspects related to $ML$-degree. We plan to study in detail the cohomological properties of $\GM$ in forthcoming articles.

We work over complex numbers $\CC$. By $\PP$ we denote the projectivisation of a vector space or of a vector bundle, i.e.~$\PP(V)=(V\setminus\{0\})/\CC^*$. In particular $\PP(N_{Y/X})$
is the exceptional divisor of the blow up of a smooth variety $X$ along a smooth subvariety $Y\subset X$, where $N_{Y/X}$ stands for the normal bundle of $Y$ in $X$. For any set $X$ we denote by $\overline{X}$ its Zariski closure.
\section*{Acknowledgements}
We would like to thank Andrzej Weber for important remarks and discussions on equivariant cohomology. We thank the referees for their insightful remarks which significantly improved the
 exposition in the paper.

\section{Maximum likelihood estimate for Gaussian models}
In this section we recall basic notions from algebraic statistics, focusing on Gaussian models. Our basic references are \cite{StUh} and \cite{SethBook}.

A multivariate Gaussian distribution, also called a normal distribution, on $V_{\RR}:=\RR^n$ is determined by the mean vector $\mu\in\RR^n$  and a positive semidefinite $n\times n$ matrix $\Sigma$, known as the \emph{covariance matrix}. In the most familiar $n=1$ case $\Sigma$ is just a positive number, which is the variance of the distribution. Equivalently, we may consider $K:=\Sigma^{-1}$, which is the \emph{concentration matrix} and is also positive definite.

Before we proceed, we would like to mention one of general aims in algebraic statistics. Given some data, which is derived from an experiment, we would like to specify the parameters of the model, in our case the vector $\mu$ and the matrix $\Sigma$ and/or $K$. For this to make sense, we need to assume that the data we gather indeed follows the Gaussian model. In our setting we will be more restrictive and assume that $K$ belongs to a fixed linear subspace $\LLL$. Such models are called \emph{linear concentration models} and were introduced by Anderson in 1970 \cite{MR0277057}. In other words:
$$\LLL=\Big\{\sum_{i=1}^a \lambda_i K_i\in S^2\RR^n\Big\}$$
for some real symmetric matrices $K_i$ and we assume that $\LLL$ intersects the cone $\PD_n$ of positive definite matrices. We obtain the \emph{cone of concentration matrices} $K_\LLL:=\LLL\cap \PD_n$.

The covariance matrices allowed by the model form a semialgebraic set $K_\LLL^{-1}=\{M^{-1}: M\in K_\LLL\}$. Both $K_\LLL$ and $K_{\LLL}^{-1}$ belong to spaces of $n\times n$ symmetric matrices. From the algebraic perspective it is more natural to assume that $K_{\LLL}^{-1}$ is in the dual space of the ambient space of $K_\LLL$. Indeed, an element of $S^2V_{\RR}$ has a natural interpretation as a linear map $V_{\RR}^*\rightarrow V_{\RR}$, hence its inverse is in the space $S^2 V_{\RR}^*\supset K_{\LLL}^{-1}$. In coordinates, the natural pairing between the two spaces is given by the trace of the product. In particular, in the ambient space of $K_{\LLL}^{-1}$ it is natural to consider $\LLL^\perp$---the vector space of linear forms that vanish on $\LLL$. The projection $\pi:S^2V_{\RR}^*\rightarrow S^2V_{\RR}^*/\LLL^\perp$ of $K_{\LLL}^{-1}$ from $\LLL^\perp$ to the space $S^2V_{\RR}^*/\LLL^\perp\simeq \LLL^*$ is a convex (most often nonpolyhedral) cone $C_\LLL$, known as \emph{the cone of sufficient statistics}. The following is a fundamental theorem in algebraic statistics. We state a version that is sufficient in our setting, providing references to more general statements.
\begin{thm}[\cite{MSUZ}, Theorem 4.4, \cite{brown1986fundamentals}, Theorem 3.6]
The cones $K_\LLL$ and $C_\LLL$ are dual to each other. The inversion of matrices followed by projection $\pi$ with center $\LLL^\perp$ is a homeomorphism of the two open cones, which factors through $K_{\LLL}^{-1}$.
\end{thm}
We next explain the importance of this theorem in statistics.

Observations give us a sequence of vectors $x_1,\dots,x_s\in\RR^n$. Classically, to estimate the mean $\mu_0\in\RR^n$ we simply take the mean of the vectors $x_i$.
Each $x_i$ gives us a sample covariance matrix $(x_i-\mu_0)(x_i-\mu_0)^T$, which is positive semidefinite. By taking their mean we obtain the \emph{sample covariance matrix} $S$, which is positive semidefinite. If $S\in K_\LLL^{-1}$ then the most
natural guess for $\Sigma$ is $S$---and this is justified by
Theorem \ref{thm:MLE} below.
However, the sample covariance matrix typically does
not lie exactly in $K_\LLL^{-1}$. 
Our aim is to find $\Sigma$ that 'best explains' our observations. To make this more precise, one maximizes the likelihood of observing data. In our case, we want to find $\Sigma$ for which the product of values of the associated distribution at the observed data points is highest possible. Such $\Sigma$ is called the \emph{maximum likelihood estimate} and abbreviated MLE. The following beautiful result tells us how to find MLE from $\pi(S)$.
\begin{thm}[\cite{brown1986fundamentals}, Theorem 5.5]\label{thm:MLE}
For linear concentration model and $\pi(S)$ in the interior of $C_\LLL$ the MLE is the unique point $\Sigma_0\in K_\LLL^{-1}$ such that $\pi(\Sigma_0)=\pi(S)$.
\end{thm}
While $\pi$ provides the bijection between $K_\LLL^{-1}$ and $C_\LLL$, the algebraic structure of this map is quite interesting. Let $X_\LLL$ be the variety that is the Zariski closure of $K_\LLL^{-1}$. General points in $X_\LLL$ are simply inverses of matrices in $\LLL$. Now $\pi$ may be regarded as a dominant map $X_\LLL\rightarrow S^2V_{\RR}^*/\LLL^\perp$. As all varieties we study are cones it is more natural to consider $X_\LLL$ as a projective variety in $\PP(S^2V_{\RR}^*)$ and $\pi$ as a rational map. It is regular on $X_\LLL$ if and only if $\PP(\LLL^\perp)\cap X_\LLL=\emptyset$. Further, from the point of view of algebraic geometry, it is natural to pass to the algebraically closed field $\CC$. This is not necessary, however it greatly simplifies the language, e.g.~when we talk about the fibers of the map. Hence, from now on $V$ will be the complex vector space obtained as complexification of $V_{\RR}$. Slightly abusing notation, we may also consider $X_\LLL$ as a complex, projective algebraic variety.
We next define the basic algebraic measure of how hard it is to compute the inverse of $\pi$.
\begin{defi}[ML-degree]
The \emph{maximum likelihood degree} of the model $\LLL$ is the degree of the dominant rational map $\pi:X_\LLL\dashrightarrow\PP(S^2V^*/\LLL^\perp)$.
\end{defi}
In other words, the maximum likelihood degree counts the number of possible complex symmetric matrices in $(\pi_{|X_\LLL})^{-1}(\pi(S))$ (which are extremal for the likelihood function) given a random sample covariance matrix $S$. The MLE is one of those preimages.
\begin{rem}
One may replace the space of symmetric matrices $S^2V$ by the space of diagonal matrices, which may be identified with $V$. Choosing a subspace $\LLL$ of this space is equivalent to choosing the model. Typically one assumes that such a subspace is defined over the integers. One considers the map defined by the gradient of the log-partition function, cf.~\cite[Example 2.4 and Example 4.3]{MSUZ}. In this example, the class of varieties $X_\LLL$ coincides with the class of (projective, not necessarily normal) toric varieties, cf.~\cite{huh2014likelihood} and references therein.  
These models are called log-linear or toric.

The MLE in this case is very interesting and has been studied a lot, with relations to moment maps \cite{clarke2018moment}, geometric modeling and linear precision \cite{Frank}, Horn parametrization \cite{JuneML}. Connections to matroid theory are particularly interesting. The cohomology ring of the resolution of the graph of the inversion map, given by the permutoherdal variety, was the fundamental object to study basic invariants of matroids \cite{JuneRev}. The ML-degree and its analogs turned out to be coefficients of associated chromatic polynomials, which lead to the proof of Heron-Rota-Welsh Conjecture \cite{JunePhD, JuneKarim}.

One of our aims and motivations was to generalize the permutohedral approach to the setting of linear concentration models. This leads us to the variety of complete quadrics, which is the analog of the permutohedral variety. Indeed, we provide interpretations of both spaces as moduli spaces. The inclusion of diagonal matrices in the space of symmetric matrices leads to the inclusion of the permutohedral variety in the variety of complete quadrics, cf.~Remark \ref{rem:inc}. The computation of ML-degree is then reduced to cohomology computation on respective varieties.
\end{rem}
In this article we will be interested in the case when $\LLL$ is a generic subspace of $S^2V$. The nongeneric cases are also very interesting leading to e.g.~graphical Gaussian models. In our case the ML-degree depends only on two numbers $a:=\dim\LLL$ and $n$. We denote it by $\phi(n,a)$.
\begin{rem}
On the projective level the map given by inversion of the matrices is the gradient of the determinant. If we restrict to diagonal maps it is the gradient of the elementary symmetric polynomial $E_{n,n}:=x_1\cdots x_n$. Both of those polynomials are hyperbolic. In general, one may study gradients of hyperbolic polynomials. This approach was taken in \cite{MSUZ}. In particular, explicit formulas for the ML-degree in case of elementary symmetric polynomials were given.

In case of $E_{n,n}$ this is simply the $\beta$ invariant of the uniform matroid of rank $a$, which is ${n-2} \choose {a-1}$. Hence, the table of all ML-degrees is in this case simply the Pascal triangle. From this perspective, $\phi(n,a)$ form a symmetric noncommutative version of the Pascal triangle.
\end{rem}
The following result, which was conjectured in \cite[Conjecture 5.5]{MSUZ}, is a straightforward corollary of a result of Teissier.
\begin{cor}\label{cor:ML}
Let $\LLL$ be a general vector subspace in a vector space $V$. Let $f$ be a homogeneous polynomial on $V$. The closure of the image of $\PP(\LLL)$ under the (rational) gradient map of $f$ in $\PP(V^*)$ is disjoint from $\PP(\LLL^\perp)$. In particular, $\phi(n,a)$ equals:
\begin{enumerate}
\item the degree of the rational map $\pi:X_\LLL \dashrightarrow \PP(\LLL^*)$;
\item the degree of the variety $X_\LLL$;
\item the degree of the rational map $\pi\circ (\cdot)^{-1}:\PP(\LLL)\dashrightarrow\PP(\LLL^*)$,
\item The coefficient of $x^{a-1}y^{{{n+1}\choose 2}-a}$ in the polynomial representing the graph of the inversion map of symmetric matrices as a cohomology class in $\PP(S^2V)\times\PP(S^2V^*)$. Here $x$ (resp.~$y$) is the class of the pull-back of a hyperplane in $\PP(S^2V^*)$ (resp.~$\PP(S^2V)$).
\end{enumerate}
Above, we identify the projection $\pi$ with its restriction to $X_\LLL$.
\end{cor}
\begin{proof}
Teissier's theorem \cite{T1}, \cite[Chapter 5]{T2} says that $\overline{(\grad f)(\PP(\LLL))}\cap\PP(\LLL^\perp)=\emptyset$ for $\LLL$ general. Then the inverse image of a point by $\pi$ is a general projective subspace that contains $\PP(\LLL^\perp)$ as a codimension one subspace. Such a subspace, over complex numbers, must intersect $X_\LLL$ in precisely the number of points equal to the degree of the variety and is the degree of $\pi$. As inversion of matrices is of degree one, point $(3)$ follows. Point $(4)$ is a restatement of point $(1)$.
\end{proof}
Formulas for $\phi(n,a)$ are quite elusive. In the theorem below point $(1)$ trivially follows from B\'ezout theorem. Point $(2)$ is an easy consequence of the classical fact that the degree of the variety of symmetric matrices of rank at most $n-2$ is ${n+1}\choose 3$. Point $(3)$ is the state of the art theorem by St\"uckrad \cite{MR1161918} and Chardin, Eisenbud and Ulrich \cite{MR3406441}.
\begin{thm}\label{thm:known} In the situation described above the following holds:
\begin{enumerate}
\item $\phi(n,a)=(n-1)^{a-1}$ for $a=1,2,3$,
\item $\phi(n,4)=\frac{1}{6}(5n-3)(n-1)(n-2),$
\item $\phi(n,5)=\frac{1}{12}(7n^2-19n+6)(n-1)(n-2)$,
\item $\phi(n,a)$ for $n\leq 6$ and any $a$ are given in \cite[Theorem 1]{StUh}.
\end{enumerate}
\end{thm}
Sturmfels and Uhler \cite[p.611]{StUh} suggest the following conjecture.
\begin{conj}\label{conj:poly}
For fixed $a$ the function $\phi(n,a)$ is a polynomial in $n$ of degree $a-1$.
\end{conj}
The computations of $\phi(n,a)$ are very challenging. By Corollary \ref{cor:ML} one could try to compute the degree of $K_\LLL^{-1}$. However, this relies on obtaining the equations of $K_\LLL^{-1}$, which is a hard implicitization problem. A more effective approach is to intersect random linear combinations of $(n-1)\times (n-1)$ minors. However, then one needs to saturate with respect to the base locus, made of matrices of rank at most $n-2$, which again is complicated.

In the next sections we present a formula for $\phi(n,a)$ given just in terms of basic arithmetic operations on rational numbers.

\section{${\mathbb C}^*$-actions and birational maps}

\subsection{Set-up and basic results}
Let us begin this section by recalling definitions and results regarding $\CC^*$-actions. Our set up is similar to that of \cite{RW} or  \cite{OSCRW}. For details we refer to \cite[Sect.~2]{RW} or \cite[Sect.~2]{OSCRW}.

Let $X$ be a smooth projective variety over $\CC$ with an effective $\CC^*$ algebraic action $\delta: \CC^*\times X\rightarrow X$; we will write $(t,x)\mapsto\delta(t)(x)$ or just $(t,x)\mapsto t(x)$ if the action is known from the context. By the {\em source} and the {\em sink} of the orbit of a point $x\in X$ we understand $\lim_{t\to 0}t(x)$ and $\lim_{t\to\infty}t(x)$, respectively. The names of source/sink will be also used in reference to the points on a rational curve which is the closure of the orbit.

By $Y_j$, $j\in\cJ$, we denote components of the fixed point locus of the action. The source of the action on $X$ is the component which is the source for an orbit of a general point, similarly with the sink.

We choose an ample line bundle $\Lb$ on $X$ and a linearization $\mu=\mu_\Lb: \CC^*\times \Lb\rightarrow\Lb$. We will write $\mu_\delta$ if we need to stress that $\mu$ lifts the action $\delta$ to $\Lb$ and $\Lb$ is known from the context. The linearization $\mu_\Lb$ assigns to each component $Y_i$ the weight on the fiber $\Lb_y$ over a point $y\in Y_i$. Thus we have a function $\mu_\Lb: \{Y_j: j\in \cJ\}\rightarrow \ZZ$. That is, by abuse of notation, $\mu_\Lb$ denotes both the linearization and the function whose value on each fixed point $y$ is the weight $\mu_\Lb(y)$ of $\mu_\Lb: \Lb_y\rightarrow\Lb_y$ and also, for a connected fixed point component $Y$, $\mu_\Lb(Y)$ denotes $\mu_\Lb(y)$ for $y\in Y$. The function $\mu_\Lb$ assumes maximum at the source and the minimum at the sink of the action. We note that if $\Lb$ is very ample then $\mu_\Lb$ is the restriction to the fixed point locus of the moment map associated to an equivariant emmbedding defined by the complete linear system of $\Lb$. For more see \cite[Sect.~2.1]{BWW} and references therein.

We will assume that the action is {\em equalized} which means that the weights of the action of the normal bundle of every fixed point component are either $+1$ or $-1$. The following lemma summarizes \cite[Lemma 2.2]{RW} and \cite[Cor.~2.15]{OSCRW}.

\begin{lemma}\label{equalized_action}
Let $\CC^*\times X\rightarrow X$ be an equalized action on $(X,\Lb)$ as above.
\begin{enumerate}[leftmargin=*]
\item If $C$ is the closure of a nontrivial orbit with source at $y_0$ and sink at $y_\infty$ then $C$ is smooth and $\Lb\cdot C =\mu_\Lb(y_0)-\mu_\Lb(y_\infty)$.
\item  If $C_1,\dots,C_m$ are closures of orbits such that the sink of $C_{i-1}$ and the source of $C_i$ are contained in the same fixed point component and the source of $C_1$ is in the source of $X$ and the sink of $C_m$ is in the sink of $X$ then the numerical equivalence class of the cycle $\sum_i[C_i]$ is equal to the class of a general orbit of the action.
\end{enumerate}
\end{lemma}

Finally, we recall that if the sink and source are of codimension one then assigning to the source of a general orbit its sink determines a birational morphism of the sink and source of the $\CC^*$-action. If the action is equalized then blowing up arbitrary sink/source (which may be just a point) yields a birational transformation of the resulting exceptional divisors; for this and more see \cite[Sect.~3]{OSCRW}. In particular we have the following corollary to 3.4 and 3.10 in \cite{OSCRW}.

\begin{cor}\label{action=>Cremona}
Let $\CC^*\times X\rightarrow X$ be an equalized action with $Y_0$ and $Y_\infty$ denoting its source and sink. Then assigning to a general point $x\in X$ the tangents to its orbit $t\mapsto t(x)$ at $Y_0$ and at $Y_\infty$ defines a birational map $\Phi: \PP(N_{Y_0/X})\dashrightarrow\PP(N_{Y_\infty/X})$.
\end{cor}


\subsection{A special $\CC^*$-action on a Lagrangian Grassmanian}\label{sec:action-LG}
Let $V$ be a $\CC$ vector space of dimension $n$ and basis $e_1,\dots, e_n$. By $f_1,\dots,f_n$ we denote the dual basis of the dual space $V^*$. The space $V\oplus V^*$ is equipped with a symplectic form $\omega=\sum_i e_i^*\wedge f_i^*$. We may think about $\omega$ in terms of the duality $\omega: V\oplus V^*\rightarrow(V\oplus V^*)^*=V\oplus V^*$ such that $\omega(f_i)=e_i$ and $\omega(e_i)=-f_i$. A subspace $U\subset V\oplus V^*$ of dimension $n$ is called Lagrangian if the restriction of the form $\omega$ to $U$ is trivial, that is $U$ is a maximal isotropic subspace with respect to $\omega$. We note the following known property of the form $\omega$: if  $\pi_V: V\oplus V^*\rightarrow V$ is the projection and $U$ is Lagrangian then
$$\pi_V(U)=\{v\in V: \forall u\in U\cap V^*\  u(v)=0\}$$
A similar observation holds for the projection $\pi_{V^*}$.

The group $SL(V)$ acts on $V$ by the standard representation and on $V^*$ by dual representation. Then the resulting action on $V\oplus V^*$ preserves the form $\omega$. The Cartan torus in $SL(V)$ whose action in the basis $e_1,\dots,e_n$ is diagonal will be denoted by $\TTT$.

We consider the diagonal action $\CC^*\times V\ni (t,v)\mapsto t\cdot v\in V$ of weight 1 and its dual $\CC^*\times V^*\rightarrow V^*$ of weight $-1$; the resulting action $\delta: \CC^*\times V\oplus V^*\rightarrow V\oplus V^*$ preserves the form $\omega$. We note that $\delta(-1)$ is just multiplication by $-1$ on $V\oplus V^*$.

 The diagonal action $\delta$ of $\CC^*$ commutes with the action of $SL(V)$ described above. In fact, we have a homomorphism with finite kernel $SL(V)\times \CC^*\rightarrow GL(V)\hookrightarrow Sp(V\oplus V^*)$ into the group of symplectic transformations of $V\oplus V^*$ which we will denote by $Sp(n)$. The homomorphism maps the product of the diagonal torus of $SL(V)$ and $\CC^*$ to the big Cartan torus of $Sp(n)$.

The representation theory of $Sp(n)=Sp(V\oplus V^*)$ can be described by the root system of type ${\mathrm C}_n$ with the Dynkin diagram
\par \centerline{\ifx\du\undefined
  \newlength{\du}
\fi
\setlength{\du}{3.3\unitlength}
\begin{tikzpicture}
\pgftransformxscale{1.000000}
\pgftransformyscale{1.000000}

\definecolor{dialinecolor}{rgb}{0.000000, 0.000000, 0.000000} 
\pgfsetstrokecolor{dialinecolor}
\definecolor{dialinecolor}{rgb}{0.000000, 0.000000, 0.000000} 
\pgfsetfillcolor{dialinecolor}


\pgfsetlinewidth{0.300000\du}
\pgfsetdash{}{0pt}
\pgfsetdash{}{0pt}

\pgfpathellipse{\pgfpoint{-6\du}{0\du}}{\pgfpoint{1\du}{0\du}}{\pgfpoint{0\du}{1\du}}
\pgfusepath{stroke}
\node at (-6\du,0\du){};

\pgfpathellipse{\pgfpoint{4\du}{0\du}}{\pgfpoint{1\du}{0\du}}{\pgfpoint{0\du}{1\du}}
\pgfusepath{stroke}
\node at (4\du,0\du){};

\pgfpathellipse{\pgfpoint{14\du}{0\du}}{\pgfpoint{1\du}{0\du}}{\pgfpoint{0\du}{1\du}}
\pgfusepath{stroke}
\node at (14\du,0\du){};

\pgfpathellipse{\pgfpoint{24\du}{0\du}}{\pgfpoint{1\du}{0\du}}{\pgfpoint{0\du}{1\du}}
\pgfusepath{stroke}
\node at (24\du,0\du){};

\pgfpathellipse{\pgfpoint{34\du}{0\du}}{\pgfpoint{1\du}{0\du}}{\pgfpoint{0\du}{1\du}}
\pgfusepath{stroke}
\node at (34\du,0\du){};

\pgfpathellipse{\pgfpoint{44\du}{0\du}}{\pgfpoint{1\du}{0\du}}{\pgfpoint{0\du}{1\du}}
\pgfusepath{stroke}
\node at (44\du,0\du){};

\pgfsetlinewidth{0.300000\du}
\pgfsetdash{}{0pt}
\pgfsetdash{}{0pt}
\pgfsetbuttcap

{\draw (-5\du,0\du)--(3\du,0\du);}
{\draw (5\du,0\du)--(13\du,0\du);}
{\draw (25\du,0\du)--(33\du,0\du);}
{\draw (34.65\du,0.7\du)--(43.35\du,0.7\du);}
{\draw (34.65\du,-0.7\du)--(43.35\du,-0.7\du);}


{\pgfsetcornersarced{\pgfpoint{0.300000\du}{0.300000\du}}\definecolor{dialinecolor}{rgb}{0.000000, 0.000000, 0.000000}
\pgfsetstrokecolor{dialinecolor}
\draw (40.8\du,-1.2\du)--(37\du,0\du)--(40.8\du,1.2\du);}

\pgfsetlinewidth{0.400000\du}
\pgfsetdash{{1.000000\du}{1.000000\du}}{0\du}
\pgfsetdash{{1.000000\du}{1.00000\du}}{0\du}
\pgfsetbuttcap
{\draw (15.3\du,-1\du)--(23\du,-1\du);}

\node[anchor=west] at (48\du,0\du){${\rm C}_n$};

\node[anchor=south] at (-6\du,1.1\du){$\scriptstyle 1$};

\node[anchor=south] at (4\du,1.1\du){$\scriptstyle 2$};

\node[anchor=south] at (14\du,1.1\du){$\scriptstyle 3$};

\node[anchor=south] at (24\du,1.1\du){$\scriptstyle n-2$};

\node[anchor=south] at (34\du,1.1\du){$\scriptstyle n-1$};

\node[anchor=south] at (44\du,1.1\du){$\scriptstyle n$};

\end{tikzpicture}
}
\noindent
We use notation of Bourbaki \cite[Planche III]{Bourbaki} which is consistent with \cite{sagemath}. In particular we take $\varepsilon_i$, $i=1,\dots,n$ as the orthonormal base of the space of characters of the Cartan torus of $Sp(n)$. The roots are $\pm 2\varepsilon_i$ and $\pm\varepsilon_i\pm\varepsilon_j$, with simple roots $\alpha_i=\varepsilon_i-\varepsilon_{i+1}$, $i=1,\dots, n-1$, $\alpha_n=2\varepsilon_n$. The fundamental weights $\overline{\omega}_i =\varepsilon_1+\cdots+\varepsilon_i$
are generators of the Weyl chamber, the cone
of dominant weights. We recall that the Weyl group is the semidirect product $\SSS_n\ltimes \langle -1\rangle^n$, 
where the group $\SSS_n$ permutes $\varepsilon_i$'s and $\langle -1\rangle^n$ changes their sign.

The embedding $SL(V)\hookrightarrow Sp(V\oplus V^*)$ is associated to the projection of the space of characters of the Cartan torus of $Sp(V\oplus V^*)$ onto the space spanned by the first $n-1$ elements of the root base $\alpha_1,\dots,\alpha_{n-1}$.

The fundamental weight $\overline{\omega}_1=\varepsilon_1$ is associated to the standard representation of $Sp(n)$ on $V\oplus V^*$ with weights $+\varepsilon_i$ and $-\varepsilon_i$ associated to vectors, respectively, $e_i$ and $f_i$. In particular, the action $\delta$ of $\CC^*$ described above is associated to projection of the space of characters $\sum a_i\varepsilon_i\mapsto\sum a_i$ with the kernel generated by the first $n-1$ simple roots $\alpha_1,\dots,\alpha_{n-1}$

On the other hand the fundamental weight $\overline{\omega}_n= \sum_1^n\varepsilon_i$  is an irreducible representation $W=W_n$ contained in $\bigwedge^n(V\oplus V^*)$. The closed orbit of $Sp(n)$ in the  projectivization $\PP(W_n)$ is the Lagrangian Grassmannian $\LG=\LG_n$ (we use the subscript if we need to indicate the dimension of $V$ and skip it otherwise) parametrizing linear Lagrangian subspaces in $V\oplus V^*$.

For $I=\{i_1<\cdots<i_r\}\subset \{1,\dots, n\}$ by $e_I$ we denote the respective wedge product $e_{i_1}\wedge\cdots\wedge e_{i_r}$; a similar notation will be used for $f_J$, $J\subset\{1,\dots,n\}$. The vector $e_I\wedge f_J\in \bigwedge^n(V\oplus V^*)$ is isotropic with respect to the form $\omega$ if and only if $J$ is the complement of $I$, that is $J=I'$.
Thus the point $[e_I\wedge f_{I'}]\in\PP(W)$ is in $\LG$ and, in fact, all fixed points of the action of the Cartan torus of $Sp(n)$ on $\LG$ are of this type. The weights associated to these points are in the Weyl group orbit of the dominant weight $\overline{\omega}_n$ and they are  $\sum_{i\in I}\varepsilon_i-\sum_{j\in I'}\varepsilon_j$. In fact, in terms of the action of the Weyl group the fundamental weight $\overline{\omega}_n$ has isotropy $\SSS_n<\SSS_n\ltimes\langle -1\rangle^n$ hence its orbit in the space of characters depends only on the change of signs of $\varepsilon_i$'s.
We note that the convex hull of the weights associated to the fixed points is a hypercube.

The weights of the action of the Cartan torus on the tangent space to $\LG$ at the Borel fixed point $[e_1\wedge\cdots\wedge e_n]$ associated to the fundamental weight $\overline{\omega}_n$ are roots which have negative intersection with $\overline{\omega}_n$. They are $-2\varepsilon_i$, $i=1,\dots,n$  and $-\varepsilon_i-\varepsilon_j$, $1\leq i<j\leq n$. Thus, in particular, $\dim\LG={{n+1}\choose{2}}$. Again, in order to write down the weights of the action on the tangent space at the point $[e_I\wedge f_{I'}]$ one has to change signs of $\varepsilon_i$'s accordingly: if $i\in I$ then the sign is negative if $i\in I'$ then the sign is positive.

\begin{ex}\label{LGn=2}
Let us consider the case $n=2$. On the  projectivisation $\PP(\bigwedge^2(V\oplus V^*))$ we introduce Pl\"ucker coordinates $x_{ij}$ associated to the ordered basis $\{e_1,e_2,f_1,f_2\}$ of $V\oplus V^*$, so that a general vector in $\bigwedge^2(V\oplus V^*)$ is
$$x_{12}\cdot e_1\wedge e_2+ x_{13}\cdot e_1\wedge f_1 + x_{14}\cdot e_1\wedge f_2+x_{23}\cdot e_2\wedge f_1+x_{24}\cdot e_2\wedge f_2 + x_{34}\cdot f_1\wedge f_2.$$
The Grassmannian of planes in $V\oplus V^*$ is the Pl\"ucker quadric
$$x_{12}\cdot x_{34}-x_{13}\cdot x_{24}+x_{14}\cdot x_{23}=0$$
while the wedge intersection with the form $\omega$ adds a linear equation
$x_{13}+x_{24}=0$ which defines a linear subspace $W\subset V\oplus V^*$. Therefore, if we set $y=x_{13}=-x_{24}$ then $\PP(W)$ has linear coordinates
$x_{12}, x_{14}, x_{23}, x_{34}, y$ on which the Cartan torus of $Sp(2)$ acts with respective weights
$$-\varepsilon_1-\varepsilon_2,\ -\varepsilon_1+\varepsilon_2,\ +\varepsilon_1-\varepsilon_2,\ \varepsilon_1+\varepsilon_2,\ 0$$
(note the change of signs which follows from the fact that we pass to the dual representation). In $\PP(W)$ the Lagrangian Grassmanian $\LG_2$ is defined by the equation
$$x_{12}\cdot x_{34} + x_{14}\cdot x_{23} + y^2 = 0$$
and we are in the situation of \cite[Ex.~2.20]{BWW}.
The point $[e_1\wedge e_2]\in \LG_2\subset \PP(W)$ is contained in two lines $x_{14}=x_{34}=y=0$ and $x_{23}=x_{34}=y=0$. Both lines are orbits of the Cartan torus action with limits at $[e_1\wedge e_2]$ and $[e_2\wedge f_1]$, $[e_1\wedge f_2]$, respectively. The weights of the torus action at the tangent space to $\LG_2$ at $[e_1\wedge e_2]$ along these lines are $-2\varepsilon_2$ and $-2\varepsilon_1$, respectively. Also the conic $x_{14}=x_{23}=x_{12}\cdot  x_{34}+y^2=0$ is the closure of an orbit of the Cartan torus with limits at $[e_1\wedge e_2]$ and $[f_1\wedge f_2]$. The weight of the torus action on the tangent to this conic at the point $[e_1\wedge e_2]$ is $-\varepsilon_1-\varepsilon_2$.
\end{ex}

In the general set-up, for arbitrary $n$, as we noted above, the action $\delta(-1)$ is multiplication by $(-1)$ in $V\oplus V^*$ and therefore $(-1)$ acts trivially on the projectivisation $\PP(W)$ hence the action descends to the quotient $\CC^*\rightarrow\CC^*/\langle -1\rangle=\CC^*$. We will call the resulting action on $\PP(W)$ by $\widehat{\delta}$. We note that in terms of characters the projection $\CC^*\rightarrow \CC^*/\langle -1\rangle$ is associated to multiplication $\ZZ\rightarrow 2\cdot\ZZ\hookrightarrow\ZZ$.

By $\Lb$ we denote the restriction of the bundle $\cO(1)$ on $\PP(W)$ to $\LG$, sections of $\cO(1)$ are linear forms on $W$. The action $\delta$ yields $t\cdot e_I\wedge f_{I'}= t^{|I|-|I'|}e_I\wedge f_{I'}$ and therefore the standard linearization of $\delta$ on the fiber of $\cO(-1)$ over $[e_I\wedge f_{I'}]$ has weight $|I|-|I'|$, see e.g.~\cite[Ex.~2.7]{RW}. Thus we can choose the linearization $\mu_\Lb$ of the action $\widehat{\delta}$ such that $$\mu_\Lb([e_I\wedge f_{I'}]) = (n-|I|+|I'|)/2$$

We summarize our discussion in the following.

\begin{lemma}\label{action_on_LG}
The action $\widehat{\delta}$ on $(\LG,\Lb)$ commutes with the action of $SL(V)$ and has the following properties:
\begin{enumerate}[leftmargin=*]
\item the source of $\widehat{\delta}$ is at $y_0=[f_1\wedge\cdots\wedge f_n]$ and sink at $y_\infty=[e_1\wedge\cdots\wedge e_n]$
\item the fixed point set of $\widehat{\delta}$ decomposes into $n+1$ components $Y_i=\Grass(n-i,V)=\Grass(i,V^*)$, for $i=0,\dots,n$; with $\mu_\Lb(Y_i)=i$; $Y_0=\{y_\infty\}, Y_n=\{y_0\}$;
\item the action $\widehat{\delta}$ is equalized.
\end{enumerate}
\end{lemma}
\begin{proof}
The first statement is obvious, see e.g.~\cite[Ex.~2.7]{RW}. For the second statement we note that the $SL(V)$-orbit of the point $[e_I\wedge f_{I'}]$ is a Grassmannian $\Grass(|I|,V)=\Grass(|I'|,V^*)$. Since all fixed points of the action of the big torus in $SL(V)$ on $\LG$ are of type $[e_I\wedge f_{I'}]$ and the action of $SL(V)$ commutes with $\widehat{\delta}$ we conclude that these are all fixed point components. For the last claim we note that the weights of the action on the tangent space of any Cartan torus fixed point are of the form $\pm 2\varepsilon_i$ or $\pm\varepsilon_i\pm\varepsilon_j$, with $i\ne j$,  with signs depending on the type  of the fixed point. Thus the weights of the action $\widehat{\delta}$ can be either $\pm 1$ or 0.
\end{proof}
By looking at the weights of the action at fixed points, as in the proof above, we conclude the following.
\begin{cor}\label{cor:small_orbits}
Let $C\subset\LG\subset\PP(W)$ be the closure of a nontrivial orbit of the action $\widehat{\delta}$ with limits at $Y_r$ and $Y_s$. Assume that $C$ is invariant with respect to the action of the Cartan torus in $SL(V)$. Then either $C$ is a line and $|r-s|=1$ or $C$ is a conic and $|r-s|=2$. In particular, if $[e_{I_1}\wedge f_{I_1'}],\ [e_{I_2}\wedge f_{I_2'}]\ \in C$ are the limit points of the orbit then $|I_1\symdif I_2|=\deg C$, where $I_1\symdif I_2$ is the symmetric difference of $I_1, I_2$.

Conversely, let $I_1,I_2\subseteq\{1,\dots,n\}$ be such that $|I_1\symdif I_2| = 1 \text{ or } 2$ and $|I_1|\ne |I_2|$. Then there exists a unique line or a conic, respectively, in $\LG$ which is a closure of an orbit of the action $\widehat{\delta}$ which contains $e_{I_1}\wedge f_{I_1'}$ and $e_{I_2}\wedge f_{I_2'}$.
\end{cor}
\begin{proof}
The sink and the source of the orbit whose closure is $C$ is a point of type $e_{I}\wedge f_{I'}$ with suitable choice of the set of indices $I$. Let $e_I\wedge f_{I'}$ be the source of the orbit. The weights of the action of the Cartan torus on the tangent space to $\LG$ at $e_I\wedge f_{I'}$ are some of the roots and they determine unique eigenvectors: each of these eigenvectors is tangent to the unique 1-dimensional orbit of the torus whose closure is either a line or a conic with respect to $\Lb$. Thus the orbit in question is one of them and the first claim follows by \ref{equalized_action}. The second statement can be proved similarly: in fact a more general version of this statement is \ref{lem:localtoglobal}.
\end{proof}

In Section \ref{sec:K-thGM} we will provide one more approach to classification of orbits of the action $\widehat{\delta}$.
In view of \ref{action=>Cremona} we get the following.

\begin{lemma}\label{action=>inversion_of_matrices}
There exists a natural identification of the tangent space at $y_0$ and $y_\infty$ to the space of $n\times n$ symmetric matrices so that the birational map $\Phi:\PP(T_{y_0})\dashrightarrow\PP(T_{y_\infty})$ defined by the action $\widehat{\delta}$ is the Cremona transformation associated to inversion of symmetric matrices or, equivalently, of symmetric tensors $\PP(S^2V)\dashrightarrow\PP(S^2V^*)$.
\end{lemma}

\begin{proof}
The arguments are standard, we recall them for convenience of the reader. We can write a vector $w\in V\oplus V^*$ in our base as $w=\sum_i a_i e_i+ b_if_i$ and therefore any $n$-linear subspace $U$ in $V\oplus V^*$ spanned by vectors $w_j$, $j=1,\dots,n$, can be presented as $n\times 2n$ matrix $[A\ \vert\ B]$ of rank $n$ with $A=[a_{ji}], B=[b_{ji}]$ square matrices such that $w_j=\sum_i a_{ji}e_i+b_{ji}f_i$.
Changing basis in $U$ is equivalent to left multiplication by elements of $GL(U)$ and therefore if both $A$ and $B$ are invertible then, as points in the Grassmannian, we have equivalence $$[A\ \vert\ B]\ \sim\ [B^{-1}\cdot A\ \vert\ I]\ \sim\ [I\ \vert\ A^{-1}\cdot B]$$ where $I$ is the identity matrix.

In these coordinates the action $\delta$ is as follows $\delta(t)[A\ \vert\ B]=[t\cdot A\ \vert\ t^{-1}\cdot B]$ hence we can write $$\widehat{\delta}(t)[I\ \vert\ A^{-1}\cdot B] = [I\ \vert\  t^{-1}\cdot A^{-1}\cdot B]\ \ {\rm and}\ \  \widehat{\delta}(t)[B^{-1}\cdot A\ \vert\ I] = [t\cdot B^{-1}\cdot A\ \vert\  I]$$ Thus the tangent to the orbit of $[A\ \vert\ B]$ at the sink $y_\infty=[I\ \vert\ 0]$ can be identified with $A^{-1}\cdot B$ while the tangent at the source $y_0=[0\ \vert\ I]$ is $B^{-1}\cdot A$. Therefore the rational map $\Phi:\PP(T_{y_0})\dashrightarrow \PP(T_{y_\infty})$ is inversion of matrices. As the tangent to Grassmannian $\Grass(n,V\oplus V^*)$ at $y_0=[V^*]$ is identified with $\Hom(V^*,(V\oplus V^*)/V^*)\iso V\otimes V$ and a similar statement is true for $y_\infty$ we get the map $\Phi: \PP(V\otimes V)\dashrightarrow\PP(V^*\otimes V^*)$.

Hence, it remains to show that if $U$ is presented by $[I\ \vert\ M]$, where $M$ is $n\times n$ matrix and $U$ is Lagrangian, i.e.~isotropic with respect to $\omega$, then $M$ is symmetric. However, by the presentation of $\omega$ introduced at the beginning of this subsection $U$ is isotropic if and only if $$0\ =\ [I\ \vert\ M]\cdot [M\ \vert -I]^\intercal\ =\ M^\intercal - M $$
hence the claim follows.
\end{proof}

\subsection{Complete Quadrics}

We will call a smooth projective variety $X$  \emph{convex} if it is covered by rational curves and for every morphism $f : \PP^1 \to X$ one has $\HH^1(\PP^1, f^*(TX)) =0$, see \cite[0.4]{FP}. The main example of convex varieties are rational homogeneous spaces $G/P$ where $G$ is semi-simple algebraic group and $P$ is its parabolic subgroup. In what follows we will fix an ample line bundle $\Lb$ on $X$.

Let us recall that the Kontsevich moduli space $\Mbar_{g,n}(X,\beta)$ parametrizes equivalence classes of stable maps of genus $g$ quasi-stable curves $C$ with $n$ marked points to an algebraic variety $X$, so that the pushforward of the fundamental class of $C$ is equal to a fixed class $\beta\in \HH_2(X,\ZZ)$. We refer to \cite[Section~1.1]{FP} for the proper definition of this moduli space. In general, the moduli space $\Mbar_{g,n}(X,\beta)$ can carry any possible singularity, but in the case when $g=0$ and $X$ is convex, the Kontsevich moduli space is much nicer. Namely, we have the following result from \cite[Thm.~2]{FP}.

\begin{thm}\label{thm:konts}
Let $X$ be convex manifold, then
\begin{enumerate}
\item the Kontsevich moduli space $\Mbar_{0,n}(X,\beta)$ is a smooth projective orbifold;
\item  points of $\Mbar_{0,n}(X,\beta)$ which represent stable maps without automorphisms are smooth.
\end{enumerate}
\end{thm}
The main result of this subsection is the following consequence of the above theorem.
\begin{thm}\label{thm:orbitsmoduli}
Let $X$ be a convex smooth projective variety with an effective  and equalized $\CC^*$-action $\alpha: \CC^*\times X\rightarrow X$. By $Y_0$ and $Y_\infty$ we denote the source and the sink of the action and by $\beta\in\HH_2(X,\ZZ)$ we denote the class of the closure of a general orbit of $\alpha$. Then there exists a smooth projective variety, the moduli space $\Mo$ which is a closed subset of $\Mbar_{0,0}(X,\beta)$ and which contains an open part parametrizing $\CC^*$-orbits in $X$ which have source at $Y_0$ and sink at $Y_\infty$.
The moduli space $\Mo$ comes with two natural morphisms $\pi_0:\Mo\to \PP(N_{Y_0/X})$ and $\pi_\infty:\Mo\to \PP(N_{Y_\infty/X})$, which resolve the birational map $\Phi_\alpha:\PP(N_{Y_0/X})\dashrightarrow\PP(N_{Y_\infty/X})$ which is associatd to $\alpha$, as defined in \ref{action=>Cremona}. In other words, the following diagram is commutative:
\[
\begin{tikzcd}
 &\Mo \arrow[rd, "\pi_\infty"] \arrow[ld, "\pi_0"']&   \\
\PP(N_{Y_0/X}) \arrow[rr,dashrightarrow,"\Phi_\alpha"]&                                   &\PP(N_{Y_\infty/X}),
\end{tikzcd}
\]
and $\pi_0, \pi_\infty$ are isomorphisms away from the indeterminacy locus of $\Phi_\alpha$.
\end{thm}

It is worthwile to note that, in fact, the variety $\Mo$ maps to  the Chow quotient for the action $\alpha$ which admits birational morphisms to GIT quotients for this action including $\PP(N_{Y_0/X})$ and $\PP(N_{Y_\infty/X})$, see \cite{Kapranov}. The above theorem assures its smoothness.

 We will realize the moduli space $\Mo$ as a component of the locus of fixed points of a $\CC^*$-action on the Kontsevich moduli space of stable maps to $X$ coming from $\alpha$.  Namely, the $\CC^*$-action $\alpha$ on $X$ lifts up to a left action on the space of stable maps to $X$. If $f: C\rightarrow X$ is a stable map then $(t\cdot f)(p)=t\cdot(f(p))$, for $p\in C$. In particular we have the action on the component $\Mbar_{0,0}(X,\beta)$ $$\widetilde{\alpha}:\CC^*\times\Mbar_{0,0}(X,\beta)\longrightarrow\Mbar_{0,0}(X,\beta)$$

Let $f:C\to X$ be a stable map to $X$, and let  $C_1,\ldots,C_r$ be irreducible components of $C$. Suppose that the class of the map $f$ in $\Mbar_{0,0}$ is fixed by the action of $\widetilde{\alpha}$. Then the image of every $C_i$ is either contained in the fixed point locus of $\alpha$ or its image is the closure of a $\CC^*$-orbit in $X$. In particular, a parametrization of the closure a general orbit of $\alpha$ is a fixed point of the action $\widetilde{\alpha}$.

Theorem~\ref{thm:orbitsmoduli} and Lemmata \ref{action_on_LG} and \ref{action=>inversion_of_matrices} allow us to make the following definition, which, together with Corollary~\ref{cor:gaussmoduli}, are crutial to our paper.

\begin{defi}
Let $\beta \in H^2(\LG, \ZZ)$ be the class of the closure of a generic $\CC^*$-orbit. We define the space  $\GM$ to be the connected component of the $\CC^*$-fixed locus of the Kontsevich moduli space $\Mbar_{0,0}(\LG,\beta)$ which contains a generic orbit.
\end{defi}

Although, the above definition of the variety of complete quadrics might sound very abstract, in Remark~\ref{rem:pointsofGM} we give a very concrete description of points of $\GM$.

\begin{rem}
There are several equivalent definitions of the $\GM$ already present in the classical literature \cite{LaksovCompleteQuadrics}. From the theoretical point of view one of the simplest is as the consecutive blow-up of (strict transforms) of loci of rank $i$ matrices for $i=1,\dots,n-2$. While this definition is very elegant, it is nontrivial to give an easy, explicit description of such strict transforms. We circumvent this problem by realizing $\GM$ as a moduli space. Further, by applying the group actions, we extract from it all the combinatorial data that we need.
\end{rem}

\begin{cor}\label{cor:gaussmoduli}
The variety $\GM$ is a smooth projective variety, which contains an open part parametrizing general $\CC^*$-orbits in $\LG$. It comes with two natural morphisms $\pi_0:\GM\to  \PP(S^2V)$ and $\pi_\infty:\GM\to  \PP(S^2V^*)$ which resolve the inversion map $\Phi: \PP(S^2V)\dashrightarrow \PP(S^2V^*)$. Moreover, $\GM$ admits an $SL(V)$-action which makes the above morphisms $SL(V)$ equivariant.
\end{cor}

We note that examples of varieties $X$ with a $\CC^*$-action as in Theorem~\ref{thm:orbitsmoduli} include $\Grass(n,2n)$ and orthogonal Grassmanian $OG(n,2n)$. Hence one has a smooth compactification of space of $\CC^*$-orbits of an equivalent of the action $\widehat{\delta}$ in these cases as well; applications of this observation will be used in forthcoming papers.

The rest of this subsection will be devoted to the proof of  Theorem~\ref{thm:orbitsmoduli}.

We fix an ample line bundle $\Lb$ on $X$ and its linearization $\mu_\Lb$. By Lemma \ref{equalized_action} the product of $\beta$ with the first Chern class of $\Lb$ is $\mu_\Lb(Y_0)-\mu_\Lb(Y_\infty)$.

\begin{lemma}\label{maps_in_Mo}
Let $\Mo\subset\Mbar_{0,0}(X,\beta)$ be an irreducible component of the fixed point locus of $\widetilde{\alpha}$ which contains a parametrization of a general orbit of $\alpha$. Suppose that $f: C\rightarrow X$ is a map in $\Mo$ with $C_i$, $i=1,\dots,r$, irreducible components of $C$. Then the following holds:
\begin{enumerate}[leftmargin=*]
\item $f(C)\cap Y_0\ne\emptyset\ne f(C)\cap Y_\infty$,
\item for $i=1,\dots,r$ the morphism $f_{|C_i}: C_i\rightarrow X$ is a parametrization of the closure of an orbit of $\alpha$,
\item after possibly renumbering $C_i$'s we have $f(C_1)\cap Y_0\ne\emptyset$, $f(C_r)\cap Y_\infty\ne\emptyset$ and for $i=1,\dots, r-1$ the intersection $f(C_i)\cap f(C_{i+1})$ consists of one point contained in a fixed point component $Y_i$ of the action $\alpha$,
\item the linearization function satisfies the following inequalities $$\mu_\Lb(Y_0)>\mu_\Lb(Y_1)>\cdots>\mu_\Lb(Y_{r-1})>\mu_\Lb(Y_\infty).$$
\end{enumerate}
\end{lemma}
\begin{proof}
First, let us note that the evaluation map for $\Mo\subset\Mbar_{0,0}(X,\beta)$ dominates $X$ and over a general point of $X$ we have the map $f\in\Mo$ which parametrizes the closure of a general orbit hence its image meets both $Y_0$ and $Y_\infty$. Thus the first statement is true for a general $f\in\Mo$ hence for every $f\in\Mo$. The remaining three points follow from Lemma \ref{equalized_action}. Indeed, we write $C=C'\cup C''$, where $f$ maps components of $C'=C_1'\cup\cdots\cup C_s'$ to orbits and components of $C''$ to fixed point components of $\alpha$. The linearization $\mu_\Lb$ yields a moment function of $f^*\Lb$ on $C$ with maximum at $\mu_\Lb(Y_0)$ and minimum at $\mu_\Lb(Y_\infty)$ and every $C_i'$ mapped to an interval bounded by the value of $\mu_\Lb$ on the source/sink of the respective orbit. Thus, because of continuity of the moment map , in view of \ref{equalized_action}, $\sum_i \deg_{C_i'} f^*\Lb\geq \mu_\Lb(Y_0)-\mu_\Lb(Y_\infty)$ and since the class of $f(C)$ is $\beta$ we conclude that this is in fact equality and components of $C''$ are mapped to points. Therefore the components $C_i'$ satisfy conditions postulated for $C_i$'s above. Finally, since $f$ is stable it cannot map a tree of (unmarked) rational curves to a point hence there are no components in $C''$.
\end{proof}
Now it remains to prove that $\Mo$ defined in Lemma \ref{maps_in_Mo} satisfies properties postulated in Theorem \ref{thm:orbitsmoduli}.
\begin{proof}
First we note that maps in $\Mo$ admit no automorphisms and therefore by \ref{thm:konts} $\Mbar_{0,0}(X,\beta)$ is smooth along $\Mo$. Therefore, by \cite{Iversen}, $\Mo$ is smooth as fixed point component of the $\CC^*$-action on the smooth locus. Finally, the maps $\pi_0$ and $\pi_\infty$ are defined by tangents to curves $f(C)$ at points of their intersection with $Y_0$ and $Y_\infty$.
\end{proof}
\begin{rem}
The core of the above proof of Theorem \ref{thm:orbitsmoduli} is the following more general observation: Suppose $X$ is a convex manifold with an equalized action $\alpha: \CC^*\times X\rightarrow X$ and let $f: \PP^1\rightarrow C\subset X$ be a parametrization of the closure of an orbit with source in $Y_1$ and sink in $Y_2$, any two fixed components of the action $\alpha$. Then the fixed point component of the action $\widetilde{\alpha}$ on $\Mbar_{0,0}(X,[C])$ which contains the class of $f$ is smooth.
\end{rem}

\begin{rem}\label{rem:pointsofGM}
The points of $\GM$ are in one to one correspondence with chains of $\CC^*$-orbits $C_1,\ldots,C_r\subset \LG$ such that 
\begin{itemize}
\item the source of $C_1$ is the source of $\LG$, the sink of $C_r$ is the sink of $\LG$;
\item the source of $C_i$ is the sink of $C_{i-1}$ for $i=2,\ldots r$.
\end{itemize}
\end{rem}

\section{Geometry of Complete Quadrics}\label{sec:K-thGM}

\subsection{Equivariant $K$-theory}\label{sec:K-th}
In this subsection we give a very brief overview of the results on equivariant $K$-theory we are going to use later. We will use the push forward formula from Theorem~\ref{thm:loc} to provide an algorithm for computing the ML-degree in Subsection~\ref{sec:gaussint}. For a more detailed introduction to the topic we refer to \cite{MR2993138, ChrissGinzburg, feher2018motivic}.

Let $X$ be a smooth complete algebraic variety of dimension $m$ with an effective action of an algebraic torus $\TTT\simeq(\CC^*)^r$. Recall that  the equivariant $K$-theory of a point with trivial $\TTT$-action is the representation ring  of $\TTT$.  That is, $K_\TTT^0(\pt)=\ZZ[M(\TTT)]\simeq \ZZ[t_1^{\pm 1},\ldots,t_r^{\pm 1}]$, where $M(\TTT)=\Hom(\TTT,\CC^*)$ is the group of characters of $\TTT$. Assume now that the action of $\TTT$ on $X$ has only finitely many fixed points. For any fixed point $p\in X^\TTT$ there is the restriction map:
\[
i^*_p : K_\TTT^0(X) \to K_\TTT^0(p).
\]
The first result we will need is the Localization theorem which states that any element in $K_\TTT^0(X)$ is determined by its images in $K_\TTT^0(p)$ for $p\in X^\TTT$.
\begin{thm}[\cite{Thomason} Theorem 2.1]\label{thm:loc1}
Let $X$ be a smooth, projective algebraic variety with a $\TTT$-action such that $X^\TTT$ is finite. Then the map 
$$
\oplus i^*_p :K_\TTT^0(X) \to\bigoplus_{p\in X^\TTT}K_\TTT^0(p)\simeq \bigoplus_{p\in X^\TTT}\ZZ[t_1^{\pm 1},\ldots,t_r^{\pm 1}]
$$
is an injection of rings.
\end{thm}
Theorem \ref{thm:loc1} states that in the situation as above any class in equivariant $K$-theory of $X$ can be uniquely represented by a collection of Laurent polynomials with integer coefficients (one for each of the fixed points of $\TTT$).

The second general statement we will need is the formula for the intersection index of classes in $K$-theory. This is done by the push forward formula. For a fixed point $p\in X^\TTT$ there is a decomposition of the cotangent space at $p$ into irreducible representations of $\TTT$:
$$
T^*_p X = \bigoplus_{i=1}^{m} \CC_{\Chi_{p,i}}
$$
We will call the collection of characters $\Chi_{p,i}$ \emph{the compass} of the action of $\TTT$ at $p$.
\begin{thm}[\cite{ChrissGinzburg} Theorem 5.11.7]\label{thm:loc}
Let $X, \TTT$ be as before and let also $f:X\to \pt$ be the trivial map to the point pt with the trivial $\TTT$-action. Finally, let $\mathcal{F}\in K_\TTT^0(X)$ be a class represented by a collection of Laurent polynomials $g_p$ with $p\in X^\TTT$. Then the pushforward $f_*(\mathcal{F})\in K_\TTT^0(\pt)$ can be computed as
\begin{equation}\label{eq:push}
f_*(\mathcal{F}) = \sum_{p\in X^\TTT}\frac{g_p}{(1-\Chi_{p,1})\ldots (1-\Chi_{p,m})}.
\end{equation}
In particular, the rational function on the right hand side of the equation (\ref{eq:push}) is a Laurent polynomial.
\end{thm}


\subsection{$K$-theory of complete quadrics} In this subsection we use results of the previous subsection to understand equivariant $K$-theory of $\GM$ with respect to the action of the Cartan torus $\TTT\iso(\CC^*)^{n-1}\subset SL(V)$ acting on $\GM$. We note that because of symmetries it is convenient to use the torus in $SL(V)$ whose action on $\LG$ and $\GM$ is not faithful (i.e.~has nontrivial kernel). Thus, eventually, we will divide the acting torus by its finite subgroup acting trivially. Our aim is achieved in two main steps: describing the $\TTT$ invariant points in $\GM$ (Proposition \ref{prop:fixedpoints}) and describing the $\TTT$ invariant rational curves in $\GM$, together with the characters with which $\TTT$ acts on them (Proposition \ref{prop:orbits}).

We begin by discussing low-dimensional cases. As in the case of $\LG=\LG_n$ we write $\GM_n$ if we need to indicate the dimension of $V$. We use notation of Section \ref{sec:action-LG}

\begin{ex}\label{GMn=2}
The Lagrangian Grassmannian for $n=2$, the 3-dimensional quadric, was presented in Example \ref{LGn=2}. We use notation from that example. The Cartan torus in $SL(V)$ is just $\CC^*$ which acts on $e_1, e_2$ with weights $1, -1$, respectively, hence its action on coordinates of $W$ has weight $0$ on $x_{12},x_{34}, y$ while on $x_{14}$ and $x_{23}$ it has weight $-2$ and $2$, respectively. Consider the projection $$W\longrightarrow W/\left(\CC\cdot(e_1\wedge e_2)+\CC\cdot(f_1\wedge f_2)\right)$$
with the quotient having coordinates $(x_{14},y,x_{23})$. The projection is equivariant for both the action $\widehat{\delta}$ which descends to the trivial action on the quotient and for the action of $\CC^*\subset PGL(V)$ which descend to the action which in the above coordinates has weights $(1,0,-1)$. Thus conics passing through $e_1\wedge e_2$ and $f_1\wedge f_2$ which are cut from $LG_2$ by planes in $\PP(W)$ are closures of orbits of $\widehat{\delta}$. Thus for $n=2$ we have $\GM=\PP^2$ with a $\CC^*$-action of weights $(-1,0,1)$ with source and sink at $x_{14}=y=0$ and $x_{23}=y=0$ representing reducible conics and an inner fixed point $x_{14}=x_{23}=0$ representing an irreducible conic, cf.~Figure \ref{fig:n=2}.
\end{ex}

\begin{ex}\label{GMn=3}
The case of a special $\CC^*$-action on the Lagrangian Grassmanian $\LG_3$ was  considered in \cite[Sect.~6]{OSCRW} where its associated birational transformation $\Phi: \PP(S^2V)\dashrightarrow\PP(S^2V^*)$ was indentified as the classical quadro-quadric Cremona transformation of $\PP^5$ with the indeterminacy center at the second Veronese $\PP^2\hookrightarrow\PP^5$ which can be resolved by a single blow-up of this center. By the discussion in the proof of \cite[Prop.~6.4]{OSCRW} it follows that the blow-up dominating both sides of the Cremona transformation is the variety of complete quadrics $\GM_3$ in this case with $\pi_0$ and $\pi_\infty$ the blow-down morphisms.

If $H_i$ and $E_i$ are the pullback of the hyperplane section divisor in each of the $\PP^5$ and the exceptional divisor, respectively, then $H_2=2H_1-E_1$ and $H_1=2H_2-E_2$ hence $\Lb=H_1+H_2=3H_i-E_i$ is ample. As the action of $SL(3)$ on $\PP^5$'s are standard or its dual we can describe the values of the function $\mu_\Lb$ on the fixed points of the action of the Cartan torus $\TTT\subset SL(V)$ on $\GM_3$.

The diagram below presents images of fixed points of this action  as $\bullet$ in the lattice of characters of $\TTT$; by $\circ$ we denote roots and some weights. The two big triangles are associated to the action of $\TTT$ on $\PP^5$'s. The arrows at one of the vertices $\circ$ of a big triangle represent weights of the conormal of $\PP^2\subset\PP^5$ at one of its fixed points  which yield the fixed points of the blow-up. The arrows at two of the fixed points $\bullet$ present their compasses in $\GM_3$.

$$ 
\begin{xy}<20pt,0pt>:
(0,0)*={\cdot}="z",
(0,1)*={\circ}="", 
(0.866,0.5)*={\circ}="w1", (1.3,0.5)*={\overline{\omega}_1}, (1.1,0.5)*={},
(0.866,-0.5)*={\circ}="w2", (1.3,-0.5)*={\overline{\omega}_2},
(0,-1)*={\circ}="e2", 
(-0.866,-0.5)*={\circ}="", 
(-0.866,0.5)*={\circ}="", 
(0.866,1.5)*={\circ}="e1-e2", (1.5,1.2)*={\varepsilon_1-\varepsilon_2},
(1.732,0)*={\circ}="e1-e3", 
(0.866,-1.5)*={\circ}="e2-e3", (1.5,-1.8)*={\varepsilon_2-\varepsilon_3},
(-0.866,-1.5)*={\circ}="e2-e1",
(-1.732,0)*={\circ}="e3-e1", 
(-0.866,1.5)*={\circ}="e3-e2",
(5.196,3)*={\circ}="S1", (0,-6)*={\circ}="S2", (-5.196,3)*={\circ}="S3",
(5.9,3)*{6\cdot\overline{\omega}_1},
"S1";"S2" **@{.}, "S3";"S2" **@{.}, "S1";"S3" **@{.},
(5.9,-3)*{6\cdot\overline{\omega}_2},
(5.196,-3)*={\circ}="T1", (0,6)*={\circ}="T2", (-5.196,-3)*={\circ}="T3",
"T1";"T2" **@{.}, "T3";"T2" **@{.}, "T1";"T3" **@{.},
(1.732,3)*={\bullet}="p1",
(3.464,0)*={\bullet}="p2", 
(1.732,-3)*={\bullet}="p3",
(-1.732,-3)*={\bullet}="p4",
(-3.464,0)*={\bullet}="p5", 
(-1.732,3)*={\bullet}="p6",
(0,3)*={\bullet}="q1",
(2.598,1.5)*={\bullet}="q2",
(2.598,-1.5)*={\bullet}="q3",
(0,-3)*={\bullet}="q4",
(-2.598,-1.5)*={\bullet}="q5",
(-2.598,1.5)*={\bullet}="q6",
"S3";"p6" **@{-} *\dir{>}, "S3";"p5" **@{-} *\dir{>}, "S3";"q6" **@{-} *\dir{>},
"p1";"q1" **@{=} *\dir{>}, "p1";"q2" **@{=} *\dir{>},
"p1";"p6" **@{-} *\dir{>}, "p1";"p2" **@{-} *\dir{>},
"p1";"e1-e2" **@{-},
"q4";"p4" **@{-} *\dir{>}, "q4";"p3" **@{-} *\dir{>},
"q4";"e2-e3" **@{-} *\dir{>}, "q4";"e2-e1" **@{-} *\dir{>},
"q4";"z" **@{-} *\dir{>},
\end{xy}
$$
The reader is advised to compare this graph with Figure \ref{fig:n=3} which presents $\TTT$-fixed points and its 1-dimensional orbits in $\GM_3$ described in terms of $\TTT$ invariant flags on $V$ which we introduce below.
\end{ex}

We start by proving preliminary lemmas that allow us to better understand the $\CC^*$-action $\widehat{\delta}$ on $\LG$.

\begin{lemma}\label{lem:sourcesink}
Let $U\subset V\oplus V^*$ be a Lagranian subspace in $V\oplus V^*$ such that $\dim \pi_V(U)=m_1$ and $\dim \pi_{V^*}(U)=m_2$, where $\pi_V$ and $\pi_{V^*}$ denote the respective projections. If $[U]\in\LG=\LG_n$ is a point representing $U$ then the source of the orbit of $[U]$ with respect to the action $\widehat{\delta}$  is $[\pi_{V^*}(U)]\in Y_{m_2}=\Grass(m_2,V^*)$ and the sink is $[\pi_V(U)]\in Y_{n-m_1}=\Grass(m_1,V)$.
\end{lemma}
\begin{proof}
As in the proof of Lemma \ref{action=>inversion_of_matrices} we represent $U$ as an $n\times 2n$ matrix $[A\ \vert \ B]$, where the rows of $A$ span $\pi_V(U)$ and the rows of $B$ span $\pi_{V^*}(U)$, and we apply the same arguments. We note that the linearization of $\widehat{\delta}$ which we use to index the components of the fixed point set is related to the standard linearization of $\delta$ by the formula $\mu_{\widehat{\delta}}=(\mu_\delta+n)/2$.
\end{proof}
As observed at the beginning of Section \ref{sec:action-LG} $\pi_{V^*}(U)\subset V^*$ is orthogonal in the sense of natural pairing to $U\cap V\subset V$. Therefore the orbit of a Lagrangian subspace $U$ determines a flag $U\cap V\subset \pi_V(U)$ in $V$ or a dual flag in $V^*$. This observation can be inverted.

\begin{lemma}\label{lem:localtoglobal}
Let us consider a flag $\mathcal{F}= (V_1\subset V_2)$ in $V)$ with $0\leq n_1= \dim V_1<n_2=\dim V_2\leq n$. Set $m=n_2-n_1$. Then there exists a natural $\CC^*$ equivariant closed embedding $\eta_{\mathcal{F}}: \LG_m\hookrightarrow \LG_n$ such that the restriction of the action $\widehat{\delta}$ on $\LG_n$ is the respective action on  $\LG_m$ and the image of the source in $\LG_m$ is the class $[V_1]\in Y_{n-n_1}=\Grass(n_1,V)$ while the image of the sink is $[V_2]\in Y_{n-n_2}=\Grass(n_2,V)$.

\end{lemma}
\begin{proof}
We take $e_1,\dots e_n$ a basis of $V$ such that $e_1,\dots, e_{n_1}$ is a basis of $V_1$ and $e_1,\dots e_{n_2}$ is a basis of $V_2$. Consider the quotient $Q=V_2/V_1$, with basis $\bar{e}_{n_1+1},\dots, \bar{e}_{n_2}$, with $\bar{e}_i$ denoting the class of $e_i$ in $V_2/V_1$, and its dual $Q^*\subset V^*_2$, the kernel of $V^*_2\rightarrow V^*_1$, with basis $f_{n_1+1},\dots,f_{n_2}$. Let $W_m\subset \bigwedge^m(Q\oplus Q^*)$ be the linear subspace with basis $\bar{e}_I\wedge f_{I'}$, where $I'$ is complement of $I$ in the set $n_1+1,\dots, n_2$, as in Section \ref{sec:action-LG}.
We define a linear map $\widetilde{\eta}_\mathcal{F}: W_m\rightarrow W_n\subset \bigwedge^n(V\oplus V^*)$ by setting $$\widetilde{\eta}_{\mathcal{F}}(\overline{e}_I\wedge f_{I'}) = e_{\widetilde{I}}\wedge f_{\widetilde{I}'}$$ where $\widetilde{I}=\{1,\dots, n_1\}\cup I$.

Recall that the action $\delta$ on $W_m$ takes weight $|I|-|I'|=2|I|-n_2+n_1$ at the vector $\overline{e}_I\wedge f_{I'}$ and $\delta$ on $W_n$ takes weight $2|I|+2n_1 -n$ on $e_{\widetilde{I}}\wedge f_{\widetilde{I}'}$.
These weights differ by the constant $n_1+n_2-n$ hence the induced embedding $\eta_\mathcal{F}:\PP(W_m)\hookrightarrow\PP(W_n)$ is $\CC^*$ equivariant. Moreover $\eta_\mathcal{F}$ maps the source of the action at $\P(W_m)$ to $[V_1]\in\LG_n$ and similarly the sink to $[V_2]\in\LG_n$. Now, to claim the embedding of $\LG_m$ in $\LG_n$ it remains to note that the orbits whose tangents generate the tangent of $\LG_m$ at the sink or source are mapped to orbits contained in $\LG_n$; this follows from Corollary \ref{cor:small_orbits}.
\end{proof}


In what follows $\TTT$-fixed points on $\GM$ will be presented as flags of $\TTT$ invariant subspaces in $V$ or in $V^*$. The key is the following observation.
\begin{prop}\label{prop:fixedpoints}
Let $f\in\GM$ be a stable map $f:C\rightarrow\LG$. Then intersection of $f(C)$ with the fixed point locus of $\widetilde{\delta}$ determines a partial flag in $V$ $$0=V_0\subset V_1\subset\cdots\subset V_{r-1}\subset V_r=V$$ with $\dim V_j= a_j$ and $[V_j]=f(C)\cap Y_{n-a_j}=f(C)\cap\Grass(a_j,V)$.

Moreover, if the stable map $f$ is $\TTT$ invariant then $a_{i+1}-a_i\leq 2$ and there exists an ascending chain of subsets $$\emptyset = I_0\subset I_1\subset \cdots \subset I_{r-1}\subset I_r=\{1,\dots,n\}=[n]$$ such that $V_j=\langle e_i: i\in I_j\rangle$, or equivalently, $f(C)\cap Y_{a_j}=[e_I\wedge f_{I'}]$. Conversely, given any ascending chain of subsets $I_i\subset [n]$, $i=0,\dots r$,  such that  $0<|I_{i+1}-I_i|\leq 2$ there exists a $\TTT$ invariant $f\in\GM$ such that $$f(C)\cap\GM^\TTT=\{[e_{I_i}\wedge f_{I_i'}]: i=0,\dots,r\}$$
\end{prop}
\begin{proof}
The first part follows by Lemma \ref{maps_in_Mo} while the second statement follows by \ref{cor:small_orbits}.
\end{proof}
From now on, we will present a $\TTT$-fixed point $p\in\GM$ as a flag in $V$ or its dual in $V^*$  or an ascending sequence of subsets in $[n]$ such that the cardinality of any two consequtive subsets differs by 2 at most.

We have the following easy consequence of Proposition \ref{prop:fixedpoints}.
\begin{cor}\label{cor:number_of_fixptsGM}
If $\varkappa(n)$ is the cardinality of the torus-invariant points $\GM_n^\TTT$ then we have:
$\varkappa(n)=n\cdot\varkappa(n-1)+{{n}\choose 2}\cdot \varkappa(n-2).$
The explicit formula is given by:
$$\varkappa(n)=n!((1+\sqrt{3})^{n+1}-(1-\sqrt{3})^{n+1})/(2^{n+1}\sqrt{3})$$
\end{cor}
The sequence $\varkappa(n)$ is the Euler characteristic of the variety of complete quadrics $X$. It appears in many different contexts, cf.~\cite{A080599}.


Our next aim is to describe weights of the action of $\TTT\subset SL(V)$ on the tangents to fixed point of its action on $\GM$. For this we will describe one dimensional $\TTT$-orbits in $\GM$. Roughly speaking, just as $\TTT$-fixed points corresponded to (reducible) curves in $\LG$, one dimensional $\TTT$-orbits in $\GM$ will determine $\TTT$-invariant surfaces in $\LG$.
Recall that $\TTT$ acts on $e_i\in V$ with weight $\varepsilon_i$. We note that in our convention $\sum \varepsilon_i=0$, as we have chosen $\TTT\subset SL(V)$. The weights of the $\TTT$ action on the tangent space will be used in Section \ref{sec:gaussint} to determine the contribution of each torus invariant point to the ML-degree.

\begin{prop}\label{prop:orbits}
Let $p\in \GM^\TTT$ be a point determined by the flag $$0=V_0\subsetneq V_1\subsetneq\dots\subsetneq V_k\subsetneq V_{k+1}=V^*$$ where $V_i=\langle f_j: j\in I_i\rangle$ for ascending chain of subsets $I_i\subset[n]$.
Then the action of $\TTT$ on the tangent space $T_p\GM$ has all characters distinct and they fall into one of the following types:
\begin{enumerate}
\item $\varepsilon_i-\varepsilon_j$ if there exists $s$ such that $i\in I_s$ and $j\not \in I_s$,
\item $2\varepsilon_i-2\varepsilon_j$ if there exists $s$ such that $I_{s+1}=I_s\cup\{i\}$, $I_{s+2}=I_{s+1}\cup\{j\}$,
\item $\varepsilon_i-\varepsilon_j$ and $\varepsilon_j-\varepsilon_i$ if there exists $s$ such that $I_{s+1}\setminus I_s=\{i,j\}$,
\item $2\varepsilon_k-\varepsilon_i-\varepsilon_j$ if there exists $s$ such that $I_{s+1}\setminus I_s=\{k\}$ and $I_{s+2}\setminus I_{s+1}=\{i,j\}$,
\item $\varepsilon_i+\varepsilon_j-2\varepsilon_k$ if there exists $s$ such that $I_{s+2}\setminus I_{s+1}=\{k\}$ and $I_{s+1}\setminus I_{s}=\{i,j\}$,
\item $\varepsilon_{i'}+\varepsilon_{j'}-\varepsilon_{i}-\varepsilon_{j}$ if there exists s such that $I_{s+2}\setminus I_{s+1}=\{i,j\}$ and $I_{s+1}\setminus I_s=\{i',j'\}$.
\end{enumerate}
In total we have ${{n+1}\choose {2}}-1$ distinct characters, which is the dimension of $\GM$.
\end{prop}
\begin{proof}
We start by proving the last claim on the number of distinct characters by induction on $n$. For $n=2$ this is a direct check, see Example \ref{GMn=2}.
For $n>2$, without loss of generality we may assume that the ascending chain of subsets $I_i$ representing $p$ starts with the following terms; in brackets we describe the first two components (counting from the source $f_{[n]}$) of the reducible curve represented by $p$.
\begin{enumerate}[(A)]
\item $I_1=\{1\}$, $I_2=\{1,2\}$, (lines $+$ line);
\item $I_1=\{1,2\}$, $I_2=\{1,2,3\}$ (conic $+$ line);
\item $I_1=\{1\}$, $I_2=\{1,2,3\}$ (line $+$ conic);
\item $I_1=\{1,2\}$, $I_2=\{1,2,3,4\}$, when $n>3$ (conic $+$ conic).
\end{enumerate}
As the cases are similar we just present the proof for case (B). By inductive step for $n-2$ we have ${{n-1}\choose {2}}-1$ characters that do not involve $I_1$. Additionally we obtain the following contributions: $2(n-2)$ from case (1) of \ref{prop:orbits}, $2$ from case (3) and $1$ from case (5). All together we have ${{n-1}\choose{2}}-1+2(n-2)+2+1={{n+1}\choose2}-1$ characters.

We have to show that each of the given characters indeed appears as a weight of a tangent vector to $\GM$ at $p$. Our construction will be explicit: we give one dimensional $\TTT$-orbits in $\GM$, which closure contains $p$ and we show that $\TTT$ acts on a given orbit with a character predicted in (1)-(6). As we know that $\GM$ is smooth by Corollary \ref{cor:gaussmoduli} this is enough to completely characterize the $\TTT$-action on the tangent space $T_p \GM$.

The orbits that appear are of two types. One type corresponding to point $(1)$ is quite similar to one dimensional torus orbits in flag varieties. These orbits do not have to be induced from the case of smaller $n$.

Consider $p,i,j$ as in point $(1)$ and without loss of generality assume $i=2$ and $j=1$. For $\lambda\in \CC$ define the automorphism of $V$ by $A_\lambda(e_1)=e_1-\lambda e_2$ and $A_\lambda(e_j)=e_j$ otherwise. We naturally consider $A_\lambda$ as an automorphism of $V\oplus V^*$. By abuse of notation, $ A_\lambda$ is also an automorphism of $\LG$ and $\GM$. We have $A_0=Id$. We may also interpret $p$ as a reducible curve in $\LG$. In particular, for any $\lambda$ the curve $A_\lambda(p)\subset \LG$ is invariant with respect to the $\CC^*$-action and represents a point in $\GM$. Note that $\lim_{\lambda\rightarrow\infty} A_\lambda(p)$ is a $\TTT$ invariant point of $\GM$. Interpreting points of $\LG$ as $n\times 2n$ matrices we have the following diagram:

\[
\begin{tikzcd}
{[a_1\,\,a_2\,\dots \,b_1\,\, b_2\,\dots] }    \arrow[r,"A_\lambda"]  \arrow[d, "{(t_1,\dots,t_n)}"']       &         {[ a_1\,\,a_2-\lambda a_1\,\dots \,b_1+\lambda b_2 \,\,b_2\,\dots]} \arrow[d, "{(t_1,\dots,t_n)}"]
  \\
{[t_1a_1\,\,t_2a_2 \,\dots \, t_1^{-1}b_1\,\,t_2^{-1}b_2\,\dots]}      \arrow[r, "{A_{\frac{t_2}{t_1}\lambda}}"']    &                                 {[t_1a_1\,\,t_2a_2-\lambda t_2a_1\,\dots\,t_1^{-1}b_1+\lambda t_1^{-1}b_2\,\,t_2^{-1}b_2\,\dots]}
\end{tikzcd}
\]
Here, $a_i$ and $b_j$ for $1\leq i,j\leq n$ are columns of the $n\times 2n$ matrices, i.e.~vectors in $\CC^n$. The distinction is that the $a$ labels correspond to the space $V$, while $b$ to the space $V^*$.
The diagram above shows that $\lambda$ parametrizes a one dimensional $\TTT$-orbit in $\GM$ with character $\varepsilon_2-\varepsilon_1$.

All of the other $1$-dimensional orbits affect only at most two consecutive irreducible components of a reducible curve representing $p\in\GM^\TTT$, i.e.~are induced from Lagrangian Grassmannians with $n$ at most four. Precisely, we use Lemma \ref{lem:localtoglobal} to induce one dimensional orbits from $\GM$ for $n=2,3,4$ to any other $\GM$. Points (2) and (3) come from the case $n=2$ as in Example \ref{GMn=2}. Points (4) and (5) come from the case $n=3$ as in Example \ref{GMn=3}. We describe in detail point (6). To fix notation we assume $s=0$, $\{i',j'\}=\{1,2\}$, $\{i,j\}=\{3,4\}$ and $n=4$. Consider the two dimensional subset of $\LG$ consisting of all points that may be represented as following $4\times 8$ matrices: 
\begin{align*}
M_{a,b}=\left[
\begin{array}{cccccccc}
1&0&0&0&0&a&0&0\\
0&1&0&0&a&0&0&0\\
0&0&1&0&0&0&0&b\\
0&0&0&1&0&0&b&0\\
\end{array}
\right],\quad a,b\in\CC^*.
\end{align*}
For fixed $a,b$ the unique $\CC^*$-orbit that passes through the point $M_{a,b}\in\LG$ represents a point $p_{a,b}\in \GM$. As the given two dimensional subset of $\LG$ is $\TTT$ invariant, so is the set $\{p_{a,b}\}_{a,b\in\C^*}=\{p_{1,\lambda}\}_{\lambda\in\CC^*}\subset \GM$. It provides a one dimensional $\TTT$-orbit in $\GM$. 
For $t=(t_1,\dots,t_4)\in\TTT$ we have:
$$t\cdot M_{1,\lambda}=   
\left[
\begin{array}{cccccccc}
t_1&0&0&0&0&t_2^{-1}&0&0\\
0&t_2&0&0&t_1^{-1}&0&0&0\\
0&0&t_3&0&0&0&0&t_4^{-1}\lambda\\
0&0&0&t_4&0&0&t_3^{-1}\lambda&0\\
\end{array}
\right]$$
$$=\left[
\begin{array}{cccccccc}
1&0&0&0&0&(t_1t_2)^{-1}&0&0\\
0&1&0&0&(t_1t_2)^{-1}&0&0&0\\
0&0&1&0&0&0&0&(t_3t_4)^{-1}\lambda\\
0&0&0&1&0&0&(t_3t_4)^{-1}\lambda&0\\
\end{array}
\right]=M_{(t_1t_2)^{-1},(t_3t_4)^{-1}\lambda}.
$$
Note that $M_{(t_1t_2)^{-1},(t_3t_4)^{-1}\lambda}$ and $M_{1,t_1t_2(t_3t_4)^{-1}\lambda}$ are in the same $\CC^*$-orbit. Hence, $t\cdot p_{1,\lambda}=p_{1,t_1t_2(t_3t_4)^{-1}\lambda}$, which is the character given in point (6).
\end{proof}

\begin{rem}\label{rem:inf}
In Proposition \ref{prop:orbits} we did \emph{not} describe all $1$-dimensional $\TTT$-orbits in $\GM$, however enough to describe the tangent space at each torus fixed point. The reason is that there are infinitely many such orbits. This is visible already in the case $n=2$ as in Example \ref{GMn=2}. \begin{figure}\label{fig:orbitsGM2}
\includegraphics[scale=0.25]{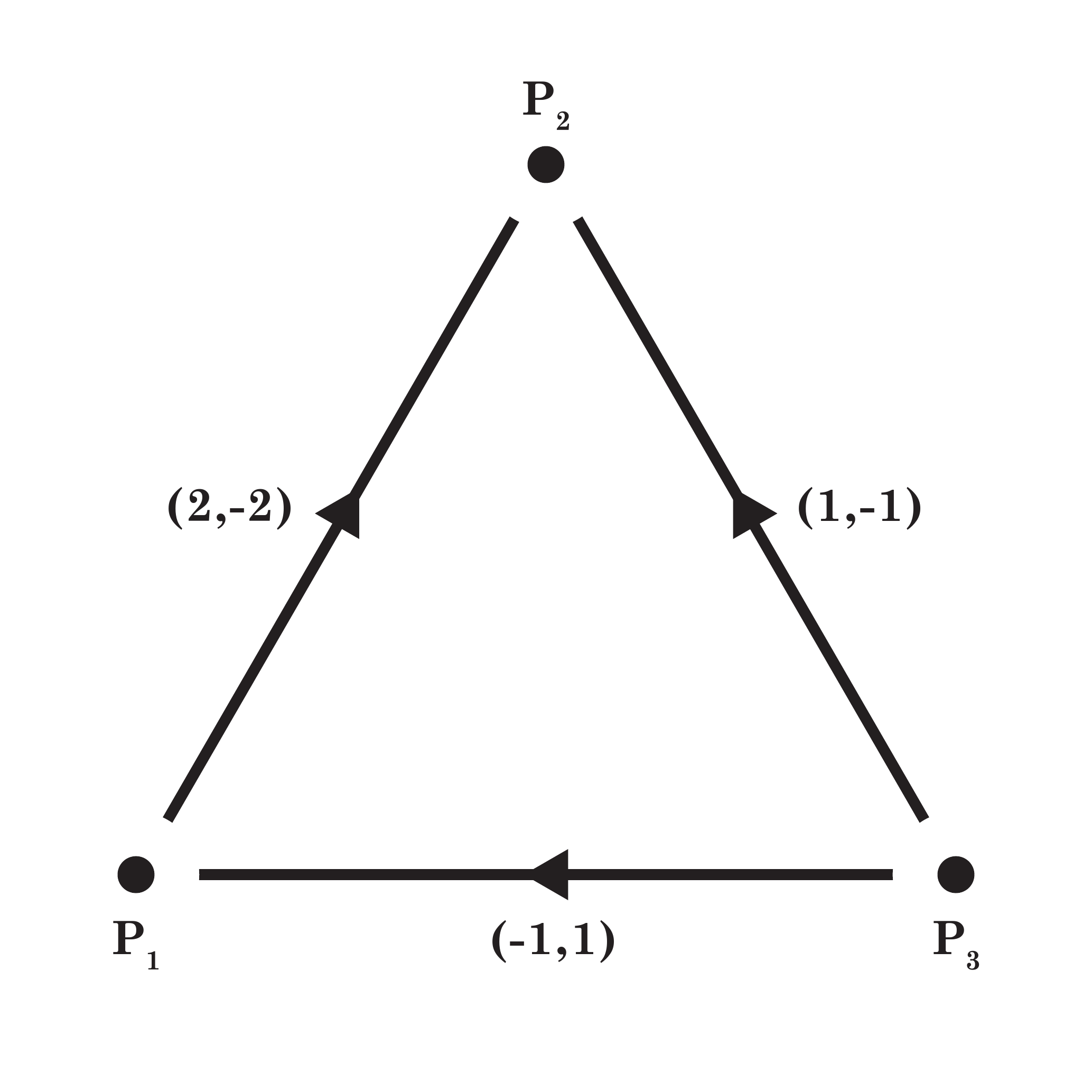}
\includegraphics[scale=0.25]{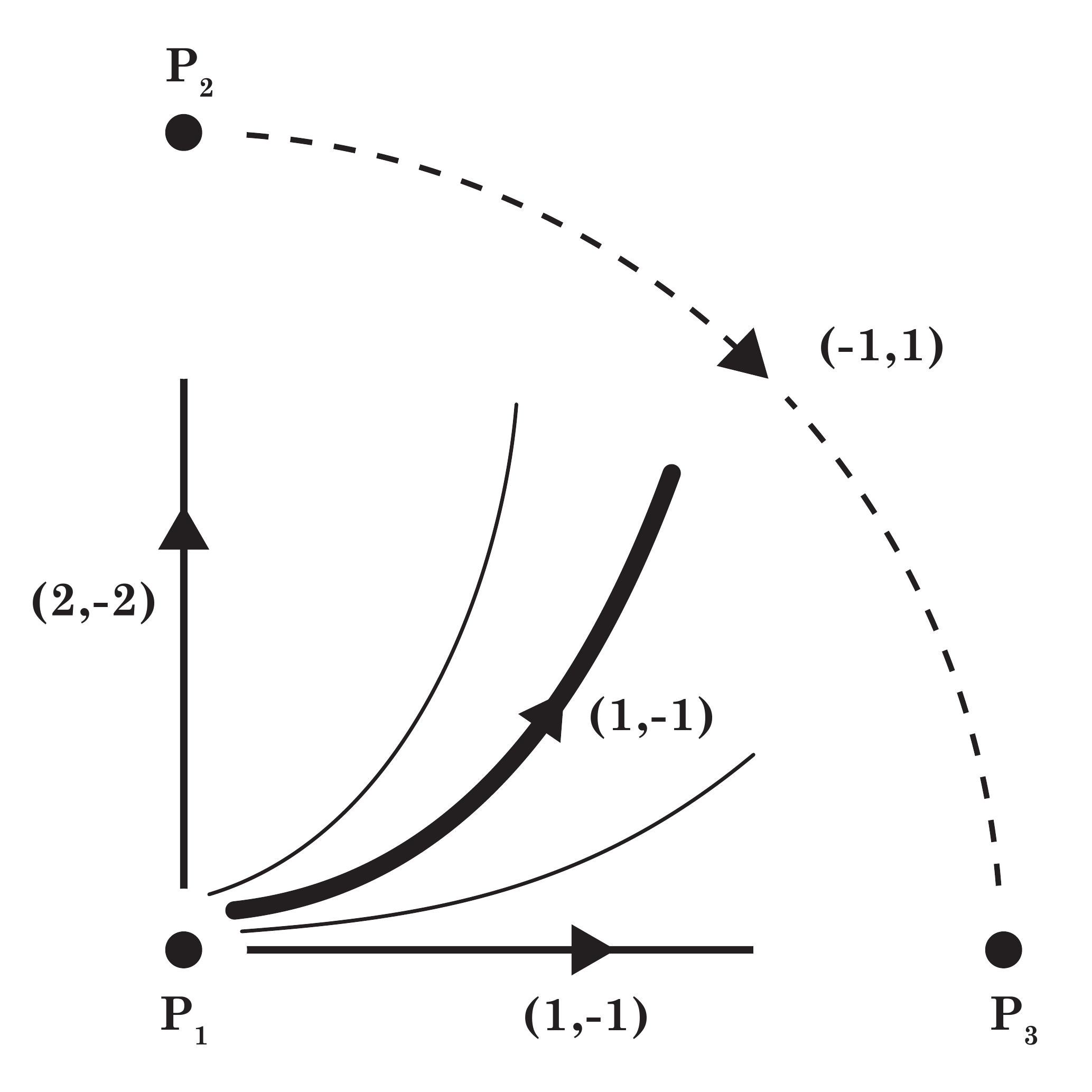}
\caption{Graph representing $\TTT$-orbits and fixed points of the complete quadrics for $n=2$ (left) and actual orbits on $\GM=\PP^2$. The vector $(a,b)$ represents the character $a\varepsilon_1+b\varepsilon_2$.}\label{fig:n=2}
\end{figure}
\end{rem}
\begin{ex}\label{ex:n=3}
In this example we consider the case $n=3$. We present the graph with vertices given by $\TTT$-fixed points in $\GM$---twelve of them---and edges corresponding to one dimensional $\TTT$-orbits---five from every vertex. (If there are infinitely many orbits joining two points, we draw just one edge, as in Remark \ref{rem:inf}.) The result is given in Figure \ref{fig:n=3}. The shaded triangles correspond to $\PP^2$, as in Remark \ref{rem:inf}, by Lemma \ref{lem:localtoglobal}.

Let us focus our attention on the set $V_0$ of six vertices corresponding to complete flags. These are exactly the torus fixed points of the moduli space $\Pe\subset\GM$ corresponding to the restriction of the inversion map to diagonal matrices. If we consider only the edges among $V_0$ that correspond to points of $\Pe$ we obtain the permutohedron, a regular hexagon in our case. This picture is general for any $n$. It represents our philosophy: the moduli of $\CC^*$-orbits on the 'correct' presentation of a rational map as a $\CC^*$-action is the 'correct' resolution of the graph of that map. Our main aim is to exploit the geometry of that moduli to obtain information about the rational map. This approach was extremely successful in case of diagonal matrices \cite{JuneKarim, JunePhD}. However, in that case the permutohedral variety was simply constructed from scratch, without reference to moduli spaces or $\CC^*$-actions.

At the same time, the study of rational maps, through $\CC^*$-actions turned out to be very advantageous \cite{Wlodarczyk}, \cite{Morelli}. We combine both points of views to efficiently perform computations of ML-degree.
\begin{figure}\label{fig:orbitsGM3}
\includegraphics[scale=0.5]{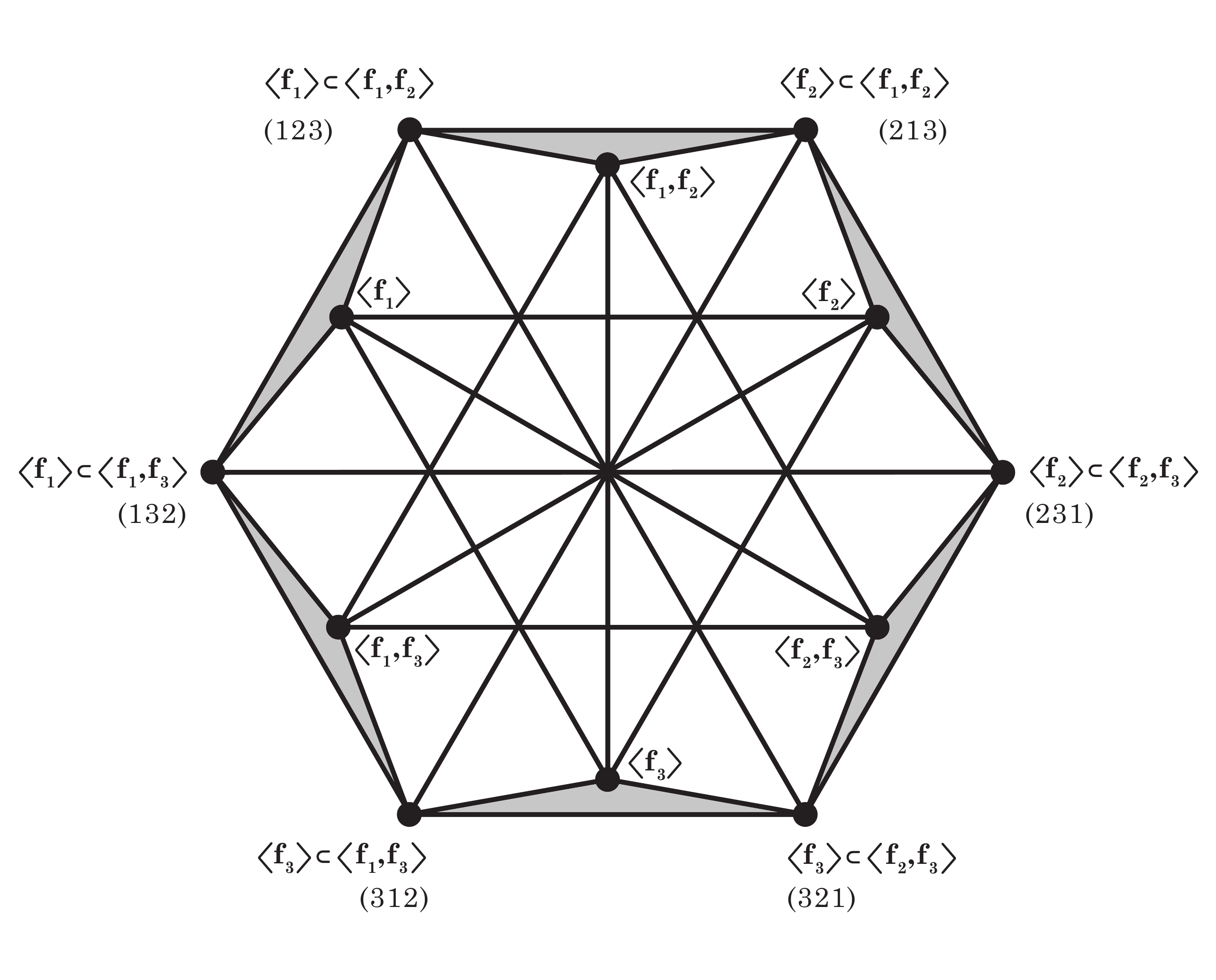}\caption{Graph representing $\TTT$-orbits and fixed points of the complete quadrics for $n=3$}\label{fig:n=3}
\end{figure}
\end{ex}
\begin{rem}\label{rem:inc}
By Corollary~\ref{cor:gaussmoduli} the points in an open set of $\GM$ are in bijection with invertible, symmetric $n\times n$ matrices. A closed subset of this open set corresponds to diagonal nondegenerate matrices. By taking its closure in $\GM$ we obtain the permutohedral variety as in Example \ref{ex:n=3}. For more detailed description of the closure of diagonal matrices in $\GM$ see \cite[Section 5.3]{deconcini1988geometry}.
\end{rem}

\subsection{Using complete quadrics for intersection}\label{sec:gaussint}
The variety of complete quadrics $\GM$ comes with two natural projections: $\pi_0:\GM\rightarrow \PP(S^2V^*)$ and $\pi_\infty:\GM\rightarrow\PP(S^2V)$, cf.~Theorem \ref{thm:orbitsmoduli}. This should be understood as associating to a $\CC^*$-orbit on the blow-up of the Lagrangian Grassmannian respectively its source and sink. Both maps are birational and isomorphism on the locus of matrices of rank at least $n-1$. Let $\Delta$ be the closure of the graph of the rational map given by inverting matrices. We have the following diagram in which $\GM$ is apparently a resolution of
the graph $\Delta$, as discussed in the
introduction:
\[
\begin{tikzcd}
                       &\GM \arrow[rdd, "\pi_\infty"] \arrow[ldd, "\pi_0"'] \arrow[d]&              \\
                      &\Delta  \arrow[rd] \arrow[ld]& \\
\PP(S^2V^*)\arrow[rr, dashrightarrow, "{A\mapsto A^{-1}}"]&             &\PP(S^2V)
\end{tikzcd}
\]

Our intersection problem now translates as follows. Given $a,b$ nonnegative integers summing up to ${{n+1}\choose{2}}-1$ compute $\phi(n,a+1)=\pi_0^{*}(H_1)^{a}\pi_\infty^{*}(H_2)^b$, where $H_1$ and $H_2$ are classes of hyperplanes respectively in $\PP(S^2V^*)$ and $\PP(S^2V)$. Suppose $p\in\GM^\TTT$ is a point given by the flag $\langle f_{i}\rangle\subset\dots\subset\langle f_1,\dots,\widehat{f_j},\dots,f_n\rangle$,  which means that the first (counting from the source) irreducible
component of the stable curve represented by $p$ is a line from the source to the point  $[e_{[n]\setminus\{i\}}\wedge f_{\{i\}}]\in\LG$. Then $\pi_0(p)$ equals $f_i\cdot f_i\in \PP(S^2(V^*))$. Similarly, if the flag begins
with $\langle f_i,f_{i'}\rangle$ (i.e.~the first component is a conic) then $\pi_0(p)=f_i\cdot f_{i'}\in\PP(S^2V^*)$. The
same applies for the last elements in the flag and the map $\pi_\infty$.
We start our computations in $\TTT$-equivariant $K$-theory of $\GM$. We use the notation from Section \ref{sec:K-th}.

We know that to represent $\mathcal{O}(1)$ on $\PP(S^2V)$ as an element of $K$-theory, as in Theorem \ref{thm:loc1}, we need to take $t_jt_{j'}$ over $e_je_{j'}\in \PP(S^2V)^{\TTT}$. Further, we have an exact sequence:
$$
0\longrightarrow\mathcal{O}(-1)\longrightarrow\mathcal{O}\longrightarrow \mathcal{O}_H\longrightarrow 0.
$$
Thus, as an element of $\TTT$-equvariant $K$-theory on $\GM$ the pull-back $\pi_\infty^{*}(H_2)$ is represented over a point $p\in\GM^\TTT$ as $1-t_jt_{j'}$ where $j,j'$ are the two elements missing in the last flag element associated to $p$ or $j=j'$ if only one element is missing. We have proved the following proposition.

\begin{prop}\label{prop:class}
Let $p\in \GM^\TTT$ be a $\TTT$-fixed point. Then, as a class in $\TTT$-equivariant $K$-theory restricted to $p\in \GM^\TTT$, the monomial $\pi_0^{*}(H_1)^a\pi_\infty^{*}(H_2)^b$ is equal to Laurent polynomial: 
$$
(1-t_i^{-1}t_{i'}^{-1})^a(1-t_jt_{j'})^b,
$$
 where $i,i'$ (resp.~$j,j'$) are the two entries (resp.~missing) in the first (resp.~last) element in the flag representing $p$. If there is only one entry (resp.~missing) we take $i=i'$ (resp.~$j=j'$). 

\end{prop}
Our next aim is to push-forward the class from the previous proposition to a point. As we know, by Proposition \ref{prop:orbits}, the compasses at each point $p\in\GM^{\TTT}$ this is straightforward from Theorem \ref{thm:loc}.
\begin{cor}\label{cor:sum}
Let $\pi:\GM\rightarrow \pt$ be the trivial map to the point $\pt$ with the trivial $\TTT$-action. The class $\pi_*(\pi_0^{*}(H_1)^a\pi_\infty^{*}(H_2)^b)$ is represented as a sum of rational functions $r_p(t_1,\dots,t_n)$ indexed by $p\in\GM^\TTT$. Using notation of Proposition \ref{prop:class} each $r_p$ is a fraction that has $(1-t_i^{-1}t_{i'}^{-1})^a(1-t_jt_{j'})^b$ in the numerator and a product $\prod_{i=1}^{{{n+1}\choose{2}}-1}(1-\tilde\Chi_i)$ in the denominator, where $\tilde\Chi_i$ are the characters in the compass of $p$ described in Proposition \ref{prop:orbits}.

The resulting sum is a Laurent polynomial. The ML-degree $\phi(n,a+1)$ equals the evaluation of that polynomial at $t_1=t_2=\dots=t_n=1$.
\end{cor}
\begin{rem}
Note that it is not possible to first substitute for $t_i$ value one in Corollary \ref{cor:sum} in the expression for $r_p$ as this would lead to zero in the denominator. At this point one needs to first sum up the rational functions and only then substitute. We will see how to avoid this complicated computation later.
\end{rem}
\begin{ex}\label{GMn=3.2} We continue Example \ref{GMn=3}
considering $n=3$. Let $p$ correspond to the flag $\langle f_1\rangle\subset \langle f_1,f_2\rangle$. Let $a=2$ and $b=3$.  The rational function $r_p$ from Proposition \ref{prop:class} equals:
$$\frac{(1-t_1^{-2})^2(1-t_3^2)^3}{(1-\frac{t_1}{t_2})(1-\frac{t_1}{t_3})(1-\frac{t_2}{t_3})(1-\frac{t_1^2}{t_2^2})(1-\frac{t_2^2}{t_3^2})}.$$
In the denominator we have a product of five functions corresponding respectively to weights: $\varepsilon_1-\varepsilon_2$, $\varepsilon_1-\varepsilon_3$, $\varepsilon_2-\varepsilon_3$,$2\varepsilon_1-2\varepsilon_2$ and $2\varepsilon_2-2\varepsilon_3$ as in Proposition \ref{prop:orbits}.
\end{ex}
Computations with rational functions in many variables, although basic, are hard. To simplify them one first considers a subtorus $\CC^*\simeq \TTT'\subset \TTT$ and passes from $\TTT$-action to $\TTT'$ action. We note that $\TTT'$ is not the $\CC^*$-action we started from. Quite the opposite we would like to first choose $\TTT'$ generic, so that we do not obtain additional fixed points on $\GM$/zeros in the denominator in the formula in Corollary \ref{cor:sum}. This means that we have a map $t_i\mapsto t^{a_i}$, where $a_i$ are some general integers. Here, explicitly general means that $a_i+a_{i'}=a_{j}+a_{j'}$ must imply $\{i,i'\}=\{j,j'\}$. Using notation from Corollary \ref{cor:sum} each $r_p$ becomes a rational function in the variable $t$, in which both the numerator and the denominator is a product of functions of type: one minus $t$ to an integral, nonzero power. The final formula for the ML-degree is now a sum of those functions (which is a Laurent polynomial in $t$) after substitution $t=1$.

At this point we have reduced the number of variables to one. However, it is still not possible to substitute $t=1$ in the formula for $r_p$ as this leads to zero in the denominator. Hence, we would still have to perform costly operation of adding rational functions. Andrzej Weber suggested us to do the computation not in equivariant $K$-theory but in equivariant cohomology \cite{AtBo, anderson2012introduction, We}. In our case it is equivalent to a very easy to prove formula:

$$\lim_{t\rightarrow 1}\frac{\prod_{i=1}^d (1-t^{b_i})}{\prod_{j=1}^d (1-t^{c_i})}=\frac{\prod_{i=1}^d b_i}{\prod_{i=1}^d c_i}.$$

Hence, although substitution $t=1$ for $r_p$ does not make sense, the limit $t\rightarrow 1$ exists and is simply a fraction of products of exponents of $t$, which are linear functions of variables $a_i$. From theory we know that after we add those limits the result is a constant, that does not depend on a choice of particular $a_i$'s and equals the ML-degree. Hence, we may simply substitute $a_i=2^i$ in each formula for $\lim_{t\rightarrow 1} r_p$. Each limit is now a number (which depends on our choice of $a_i$'s) and so is the sum (which does not depend on the choice). This computation may be performed very efficiently and the complexity of computation is the number of points in $\GM^\TTT$. By Corollary \ref{cor:number_of_fixptsGM} this complexity is of the order of $n!$, which of course is large. Still, our implementations work well for many $n$, going beyond results known so far.

\section{Algorithms, computations and conclusions}
We have implemented two algorithms to compute the ML degree $\phi(n,a)$. One is written in Sage \cite{sagemath} and one in Macaulay2 \cite{M2}. The implementations are available at \cite{web}.
As adding rational functions is computationally hard, both algorithms apply the reduction to operations on numbers described at the end of Section \ref{sec:K-thGM}.

The major steps for the algorithms are:
\begin{enumerate}
\item Creating the moment graph $G$ representing the $\TTT$-equivariant $K$-theory of the complete quadrics;
\item Summing over the vertices of the graph $G$ to obtain the value of $\phi(n,a)$. 
\end{enumerate}
The main difference among the programs is step $(1)$. In Sage we use the recursive formula for the construction of $G$ on $n$. For small $n$ we explicitly input the graph and for larger $n$ we proceed as described in Lemma \ref{lem:localtoglobal}. Hence, this option is best if we want to compute $\phi(n,a)$ for all values of $a$ and bounded $n$. 

In Macaulay2 we apply directly the build-in command to list all partitions of $n$ with size at most two and for each partitions we directly list all possible corresponding vertices. 

In both cases we remember the sequence of characters associated to a given vertex of $G$, described in Proposition \ref{prop:orbits}. The programs are available on \cite{sagemath}. The Sage version directly outputs $\phi(n,a)$. The Macaulay2 version is called by typing $phi(n,a+1)$ for any integers $a,n$. 

We hope that our approach will lead to further advances on the ML degree. However, currently the programs are very useful for making predictions. For example we have the following conjecture.

\begin{conj}\label{conj:form}
$$\phi(n,6)=\frac{1}{120}(n-1)(n-2)(43n^3-221n^2+316n-60)$$
$$\phi(n,7)=\frac{1}{60}(n-1)(n-2)(n-3)(12n^3-57n^3+81n-10)$$
$$\phi(n,8)=\frac{1}{840}(n-1)(n-2)(n-3)(87n^4-654n^3+1755n^2-1844n+140)$$
$$\phi(n,9)=\frac{1}{3360}(n-1)(n-2)(n-3)(169n^5-1770n^4+7163n^3-14042n^2+12136n-560)$$
\[\phi(n,10)=\frac{1}{362880}(n-1)(n-2)(n-3)\]\[(8357n^6-114126n^5+629471n^4-1816902n^3+2911016n^2-2201088n+60480)\]
\[\phi(n,11)=\frac{1}{907200}(n-1)(n-2)(n-3)(n-4)\]
\[(9053n^6-118395n^5+625700n^4-1749975n^3+2847707n^2-2352810n+37800)\]
\[\phi(n,12)=\frac{1}{9979200}(n-1)(n-2)(n-3)(n-4)\]\[(40993n^7-685483n^6+4763290n^5-17995750n^4+41239027n^3-59728927n^2+45442410n-415800)\]
\end{conj}
It is very easy to derive Conjecture \ref{conj:form} from Conjecture \ref{conj:poly} using our algorithms. We are also convinced that our methods allow, at the very least, to prove the formula for $\phi(n,6)$. It is interesting to note that in all known cases $\phi(0,a+1)=(-1)^a$ and $\phi(-1,a+1)=(-2)^a$ when we extrapolate it as a polynomial. It is easy to understand the term $(n-1)(n-2)\cdots$ in $\phi(n,a+1)$, as the function must vanish for fixed $a$ for small values of $n$, precisely when $a\geq{{n+1}\choose{2}}$. The remaining polynomials are quite mysterious. Comparing Theorem \ref{thm:known} and Conjecture \ref{conj:form} we see that their coefficients are alternating in signs, with log-concave absolute values, but the polynomials are not always real rooted. 

\begin{question}
To which complexity class does the problem of computing $\phi(n,a)$ belong?
\end{question}

\bibliography{bibML}
\bibliographystyle{plain}
\end{document}